\documentclass{amsart}
\usepackage[utf8]{inputenc}
\usepackage{amsmath}
\usepackage{amssymb}
\usepackage{mathrsfs}
\usepackage{graphicx}

\usepackage[colorlinks]{hyperref}
\usepackage[svgnames, dvipsnames, usenames]{xcolor}
\hypersetup{urlcolor = RedViolet, linkcolor = RoyalBlue, citecolor = ForestGreen}

\usepackage{tikz}
\usetikzlibrary{matrix}
\usetikzlibrary{arrows}

\usepackage{partmac}

\theoremstyle{plain} 
\newtheorem{thm}{Theorem}[section]
\newtheorem{prop}[thm]{Proposition}
\newtheorem{lem}[thm]{Lemma}
\newtheorem{cor}[thm]{Corollary}
\theoremstyle{definition}
\newtheorem{defn}[thm]{Definition}
\newtheorem{rem}[thm]{Remark}
\newtheorem{ex}[thm]{Example}

\numberwithin{equation}{section}

\theoremstyle{plain}
\newtheorem{innercustomthm}{Theorem}
\newenvironment{customthm}[1]
  {\renewcommand\theinnercustomthm{#1}\innercustomthm}
  {\endinnercustomthm}

\newenvironment{sproof}{%
  \proof}{\endproof}

\renewcommand{\theta}{\vartheta}
\renewcommand{\phi}{\varphi}
\renewcommand{\epsilon}{\varepsilon}
\renewcommand{\subset}{\subseteq}
\renewcommand{\supset}{\supseteq}

\newcommand{\ttimes}{\mathbin{\ooalign{$\times$\cr\kern.2ex\raise.2ex\hbox{$\times$}}}}
\newcommand{\timess}{\mathbin{\ooalign{\kern.2ex$\times$\cr\raise.2ex\hbox{$\times$}}}}

\newcommand{\N}{\mathbb N}
\newcommand{\Z}{\mathbb Z}
\newcommand{\T}{\mathbb T}

\newcommand{\R}{\mathbb R}
\newcommand{\C}{\mathbb C}

\newcommand{\norm}[1]{\lVert {#1}\rVert}
\DeclareMathOperator{\Alt}{Alt}
\DeclareMathOperator{\GL}{GL}
\DeclareMathOperator{\Mor}{Mor}

\DeclareMathOperator{\spanlin}{span}
\DeclareMathOperator{\Lrot}{Lrot}
\DeclareMathOperator{\Rrot}{Rrot}

\DeclareMathOperator{\id}{id}
\DeclareMathOperator{\Tr}{Tr}
\DeclareMathOperator{\Irr}{Irr}
\DeclareMathOperator{\lcm}{lcm}
\newcommand{\Cat}{\mathscr{C}}
\newcommand{\Kat}{\mathscr{K}}
\newcommand{\W}{\mathscr{W}}

\newcommand{\Part}{\mathscr{P}}
\newcommand{\nlin}{_{N\mathchar `\-\mathrm{lin}}}

\newcommand{\tiltimes}{\mathbin{\tilde\times}}
\newcommand{\tilstar}{\mathbin{\tilde *}}

\def\globcol{
\Partition{
\Pline (2,1)(2,0)
\Pline (3,1)(3,0)
\Ppoint1 \Pw:2
\Ppoint1 \Pb:3
\Ppoint0 \Pb:2
\Ppoint0 \Pw:3
}}

\def\ppen{\penalty300 }
\let\col=\colon
\def\colon{\col\ppen}

\begin{document}
\title{Gluing compact matrix quantum groups}
\author{Daniel Gromada}
\address{Saarland University, Fachbereich Mathematik, Postfach 151150,
66041 Saarbr\"ucken, Germany}
\email{gromada@math.uni-sb.de}
\date{\today}
\subjclass[2010]{20G42 (Primary); 18D10, 46L65 (Secondary)}
\keywords{compact matrix quantum group, representation category, complexification, gluing}
\thanks{The author was supported by the collaborative research centre SFB-TRR~195 ``Symbolic Tools in Mathematics and their Application''. The article is a~part of the authors PhD thesis.}
\thanks{I would like to thank to my PhD supervisor Moritz Weber for numerous comments and suggestions regarding the text. I also thank to Adam Skalski and Pierre Tarrago for inspiring discussions on the topic.}

\begin{abstract}
We study glued tensor and free products of compact matrix quantum groups with cyclic groups -- so-called tensor and free complexifications. We characterize them by studying their representation categories and algebraic relations. In addition, we generalize the concepts of global colourization and alternating colourings from easy quantum groups to arbitrary compact matrix quantum groups. Those concepts are closely related to tensor and free complexification procedures. Finally, we also study a more general procedure of gluing and ungluing.
\end{abstract}

\maketitle
\section*{Introduction}

The subject of this article are compact matrix quantum groups as defined by Woronowicz in~\cite{Wor87}. A~lot of attention has recently been devoted to quantum groups possessing a~combinatorial description by \emph{categories of partitions}. Those were originally defined in \cite{BS09}. Since then, their full classification was obtained~\cite{RW16} and many generalizations were introduced \cite{CW16,Fre17,TW17,Ban18super,GW19}. 

Studying and classifying categories of partitions or generalizations thereof is useful for the theory of compact quantum groups for several reasons. The primary motivation is finding new examples of quantum groups since every category of partitions induces a~compact matrix quantum group. Those are then called \emph{easy} quantum groups. In addition, since the categories of partitions are supposed to model the representation categories of quantum groups, we immediately have a~lot of information about the representation theory of such quantum groups (see also \cite{FW16}).

Categories of partitions provide a~particularly nice way to describe the representation categories of quantum groups. Understanding the structure of partition categories and obtaining some classification results, we may obtain analogous statements also for the associated quantum groups. Such results can then be generalized and go beyond categories of partitions and easy quantum groups. Let us illustrate this on a~few examples.

The classification of ordinary non-coloured categories of partitions involves a~special class of so-called \emph{group-theoretical} categories. Those induce \emph{group-theoretical} quantum groups described by some normal subgroups $A\trianglelefteq\Z_2^N$. However, the latter definition turns out to be more general~-- not every group-theoretical quantum group can be described by a~category of partitions~\cite{RW15}. Another example are glued products, which were defined in~\cite{TW17} in order to interpret some classification result on two-coloured categories of partitions. The definition of glued products was inspired by partitions, but it is independent of the partition description. Last example comes from \cite{GW19}, where a~certain classification result for so-called categories of partitions with extra singletons was obtained. Some of the new categories were interpreted by some new $\Z_2$-extensions of quantum groups. In addition, this extension procedure was generalized to a~new product construction interpolating the free and the tensor product of quantum groups.

The goal of the current paper is to study some additional results obtained in previous works on coloured partition categories \cite{TW17,TW18,Gro18,GW19}. We reformulate those results purely in terms of quantum groups and their representation categories without referring to partitions. Below, we give a~detailed overview of the results of this paper. Let us start by recalling the above mentioned glued product construction.

Consider a~compact matrix quantum group $G=(C(G),v)$ and a cyclic group dual $\hat\Z_k=(z,C^*(\Z_k))$ (for $k=0$, we take $\hat\Z_k:=\hat\Z=\T$). We can construct the tensor product or the free product as
$$G\times\hat\Z_k=(v\oplus z,C(G)\otimes_{\rm max}C^*(\Z_k)),\qquad G*\hat\Z_k=(v\oplus z,C(G)*_{\C}C^*(\Z_k)).$$
For both quantum groups, we can consider the representation $vz=v\otimes z$. This representation may not be faithful, so it defines a quotient quantum group called the \emph{glued product}
$$G\tiltimes\hat\Z_k=(vz,C(G\tiltimes\hat\Z_k)),\qquad G\tilstar\hat\Z_k=(vz,C(G\tilstar\hat\Z_k)).$$
The glued tensor product with $\hat\Z_k$ is also called the \emph{tensor $k$-complexification}. Likewise the glued free product is called the \emph{free $k$-complexifiation}. As we already mentioned, this definition comes from \cite{TW17}.

Another important concept appearing throughout the whole article is the \emph{degree of reflection} -- defined in \cite{TW17} for categories of partitions, generalized in \cite{GW19} for arbitrary quantum groups, and further characterized in Sect.~\ref{secc.degref} of this article.

The main topic of this article is the characterization of the tensor and free complexifications. Let us start with the tensor case. The following theorem characterizes the tensor complexification in terms of algebraic relations, the associated representation category, and topological generation.

\begin{customthm}{A}[Theorem \ref{T.tenscomp2}]\label{T.A}
Consider a~compact matrix quantum group $G=(C(G),v)$, $k\in\N_0$. Denote by~$z$ the generator of $C^*(\Z_k)$ and by $u:=vz$ the fundamental representation of $G\tiltimes\hat\Z_k$. We have the following characterizations of $G\tiltimes\hat\Z_k$.
\begin{enumerate}
\item The ideal $I_{G\tiltimes\hat\Z_k}$ of algebraic relations in $C(G\tiltimes\hat\Z_k)$ is the $\Z_k$-homogeneous part of the ideal $I_G$ corresponding to $G$.
\item The representation category of $G\tiltimes\hat\Z_k$ looks as follows
$$\Mor(u^{\otimes w_1},u^{\otimes w_2})=\begin{cases}\Mor(v^{\otimes w_1},v^{\otimes w_2})&\text{if $c(w_2)-c(w_1)$ is a~multiple of $k$,}\\\{0\}&\text{otherwise.}\end{cases}$$
\item The quantum group $G\tiltimes\hat\Z_k$ is topologically generated by $G$ and $\hat\Z_k$.
\end{enumerate}
\end{customthm}

We also generalize the concept of \emph{global colourization} introduced in \cite{TW18} for categories of partitions. It turns out that this concept characterizes the tensor $k$-complexification of orthogonal quantum groups for $k=0$. The case $k\in\N$ remains open.
\begin{customthm}{B}[Theorem \ref{T.tenscomp}]
\label{T.B}
Consider $G\subset U^+(F)$ with $F\bar F=cI$, $c\in\R$. Then $G$ is globally colourized with zero degree of reflection if and only if $G=H\tiltimes\hat\Z$, where $H=G\cap O^+(F)$.
\end{customthm}

The above two theorems can be understood as a~generalization of the work \cite{Gro18} on globally-colourized categories of partitions to arbitrary quantum groups. In addition, we provide the irreducible representations of tensor complexifications in Proposition~\ref{P.tiltimesirrep}.

We continue by studying the free complexification. In this case, we do not have many results even for easy quantum groups. In the recent work \cite[Sec.~4.3]{GW19}, partitions with alternating colouring were introduced and linked to free $k$-complexifications for $k=2$. In this article, we show that the free $k$-complexification actually often do not depend on the number $k$. The following two results form an analogy to Theorems \ref{T.A} and~\ref{T.B}.

\begin{customthm}{C}[Theorem \ref{T.freecompid}]
\label{T.C}
Let $H$ be a~compact matrix quantum group with degree of reflection $k\neq 1$. Then all $H\tilstar\hat\Z_l$ coincide for all $l\in\N_0\setminus\{1\}$.
\begin{enumerate}
\item The ideal $I_{H\tilstar\hat\Z_l}$ of algebraic relations in $H\tilstar\hat\Z_l$ is generated by the alternating polynomials in~$I_H$.
\item The representation category $\Cat_{H\tilstar\hat\Z_l}$ corresponding to $H\tilstar\hat\Z_l$ is a~(wide) subcategory of the representation category $\Cat_H$ generated by the sets $C(\emptyset,(\wcol\bcol)^j):=\Cat_H(\emptyset,(\wcol\bcol)^j)$, $j\in\Z$.
\end{enumerate}
The above characterization holds also for $k=1$ and $l=0$.
\end{customthm}

\begin{customthm}{D}[Theorem \ref{T.freecompchar}]
\label{T.D}
Consider $G\subset U^+(F)$ with $F\bar F=c\,1_N$. Then $G$ is alternating and invariant with respect to the colour inversion if and only if it is of the form $G=H\tilstar\hat\Z$, where $H=G\cap O^+(F)$.
\end{customthm}

Finally, the glued tensor product and the glued free product, which are used to define the quantum group complexifications, can be also understood as a special case of some \emph{gluing procedure}. In Section \ref{sec.unglue}, we ask whether we can perform this procedure in the converse direction and \emph{unglue} some cyclic group $\hat\Z_k$ from a unitary quantum group. A particularly nice result can be obtained if we are ungluing $\hat\Z_2$. It is a generalization of \cite[Theorem 4.10]{GW19}.

\begin{customthm}{E}[Theorem \ref{T.gluecorr}]
\label{T.E}
There is a~one-to-one correspondence between
\begin{enumerate}
\item quantum groups $G\subset O^+(F)*\hat\Z_2$ with degree of reflection two and
\item quantum groups $\tilde G\subset U^+(F)$ that are invariant with respect to the colour inversion.
\end{enumerate}
This correspondence is provided by gluing and canonical $\Z_2$-ungluing.
\end{customthm}

In addition, we characterize coamenability and provide irreducible representations of the canonical $\Z_2$-ungluings. These results are applied to the new $\Z_2$-extensions introduced recently in \cite{GW19} as those are special examples of canonical $\Z_2$-ungluings.

\section{Preliminaries}
\label{sec.prelim}

\subsection{Graded algebras}
Through the whole paper, we denote by $\Z_k$ the cyclic group of order $k\in\N_0$ putting $\Z_k:=\Z$ for $k=0$.

A $\Z_k$-grading of a~$*$-algebra~$A$ is a~decomposition of the algebra into a~vector space direct sum
$$A=\bigoplus_{i\in\Z_k}A_i$$
such that the multiplication and involution of the algebra respect the operation on~$\Z_k$, that is,
$$A_iA_j\subset A_{ij},\quad A_i^*\subset A_{-i}.$$
The elements of the $i$-th part $a\in A_i$ are called \emph{$\Z_k$-homogeneous of degree~$i$}.

By definition, every element $f\in A$ uniquely decomposes as $f=\sum_{i\in\Z_k}f_i$ with $f_i\in A_i$. We call the elements~$f_i$ the \emph{homogeneous components} of~$f$.

An ideal $I\subset A$ is called \emph{$\Z_k$-homogeneous} if it contains with every element~$f$ all its homogeneous components~$f_i$. A~quotient of the algebra with respect to a~homogeneous ideal inherits the grading.

The definition of a~$\Z_k$-grading for C*-algebras is quite simple for $k\in\N$. In the case of the group $\Z$ or other groups, it gets a bit complicated and we will not mention it here. Let $A$ be a~C*-algebra. A~$\Z_k$-grading on $A$ is defined by a~\emph{grading automorphism}, that is, an automorphism $\alpha\colon A\to A$ satisfying $\alpha^k$. Its spectrum consists of $k$-th roots of unity and the corresponding eigenspaces can be identified with the homogeneous parts of $A$ satisfying the properties of the algebraic definition above.

If $A$ is a~$\Z_k$-graded $*$-algebra by the algebraic definition, we can define the grading automorphism by setting $\alpha(x)=e^{2\pi\,ij/k}x$ for $x\in A_k$. The grading automorphism can be then extended to the C*-envelope $C^*(A)$ by the universal property.

\subsection{Compact matrix quantum groups}
\label{secc.qgdef}
We provide here only a~brief overview of the notions concerning compact matrix quantum groups. For more information, see for example \cite{NT13,Tim08}.

Let $A$ be a C*-algebra, $u\in M_N(A)$, $N\in\N$. The pair $(A,u)$ is called a \emph{compact matrix quantum group} if
\begin{enumerate}
\item the elements $u_{ij}$ $i,j=1,\dots, N$ generate $A$,
\item the matrices $u$ and $u^t=(u_{ji})$ are invertible,
\item the map $\Delta\colon A\to A\otimes_{\rm min} A$ defined as $\Delta(u_{ij}):=\sum_{k=1}^N u_{ik}\otimes u_{kj}$ extends to a $*$-homomorphism.
\end{enumerate}

The $*$-subalgebra $O(G)$ generated by the elements $u_{ij}$ is dense in $A$ and generalizes the coordinate ring of $G$. It is actually a Hopf $*$-algebra, that is, it is closed under the above defined \emph{comultiplication} $\Delta$, it is further equipped with a \emph{counit} (a~$*$-homomorphism $\epsilon\colon O(G)\to\C$ mapping $u_{ij}\mapsto\delta_{ij}$), and an \emph{antipode} (an antihomomorphism mapping $u_{ij}\mapsto [u^{-1}]_{ij}$).

Two compact matrix quantum groups $G=(A,u)$ and $G'=(A',u')$ are considered to be {\em identical} if there is a $*$-isomorphism $O(G)\to O(G')$ mapping $u_{ij}\to u_{ij}'$. Note that the C*-algebras $A$ and $A'$ might not be isomorphic -- those might be two different completions of $O(G)$. To overcome this ambiguity, we work with the universal C*-algebra $C_{\rm u}(G):=C^*(O(G))$. Given a compact matrix quantum group $G=(A,u)$, its {\em maximal version} $G=(C_{\rm u}(G),u)$ is again a compact matrix quantum group. From now on, we will assume that every quantum group appearing in the paper is in its maximal version and denote $C(G):=C_{\rm u}(G)$.

The above notion of identical compact matrix quantum groups indeed generalizes the notion of matrix groups being the same (i.e. not only isomorphic, but also represented by the same matrices). Similarly, we can define $H=(C(H),v)$ to be a \emph{quantum subgroup} of $G=(C(G),u)$ if there is surjective $*$-homomorphism $O(G)\to O(H)$ (or $C(G)\to C(H)$) mapping $u_{ij}\mapsto v_{ij}$ assuming both matrices $u$ and $v$ have the same size.

On the other hand, two compact (matrix) quantum groups $G$ and $H$ are said to be \emph{isomorphic}, denoted $G\simeq H$, if there exists any $*$-isomorphism  $\phi\colon C(G)\to C(H)$ such that $\Delta_H\circ\phi=(\phi\otimes\phi)\circ\Delta_G$.

An important question is also how to \emph{construct} quantum subgroups. A~set $I\subset O(G)$ is called a~\emph{coideal} if
$$\Delta(I)\subset I\otimes O(G)+O(G)\otimes I\quad\text{and}\quad\epsilon(I)=0.$$
A coideal that is also a~$*$\hbox{-}ideal is called a~\emph{$*$\hbox{-}biideal}. A~\emph{Hopf $*$\hbox{-}ideal} is a~$*$\hbox{-}biideal that is invariant under the antipode, that is, $S(I)\subset I$. Hopf $*$\hbox{-}ideals are in a~one-to-one correspondence with quantum subgroups. That is, given $H\subset G$, the kernel of the surjective $*$\hbox{-}homomorphism $O(G)\to O(H)$ is a~Hopf $*$\hbox{-}ideal. Conversely, given any Hopf $*$\hbox{-}ideal $I\subset O(G)$, then the quotient $O(H):=O(G)/I$ is a Hopf algebra that defines a~quantum subgroup $H\subset G$.


\subsection{Representations of CMQGs}

For a compact matrix quantum group $G=(C(G),u)$, we say that $v\in M_n(C(G))$ is a representation of $G$ if $\Delta(v_{ij})=\sum_{k}v_{ik}\otimes v_{kj}$, where $\Delta$ is the comultiplication defined in the previous subsection. In particular, the matrix $u$ is a representation called the \emph{fundamental representation}.

A representation $v$ is called \emph{non-degenerate} if it is invertible as a matrix, it is called \emph{unitary} if it is unitary as a matrix, i.e. $\sum_k v_{ik}v_{jk}^*=\sum_k v_{ki}^*v_{kj}=\delta_{ij}$. Two representations $v$ and $w$ are called \emph{equivalent} if there is an invertible matrix $T$ such that $vT=Tw$.

For every compact quantum group it holds that every non-degenerate representation is equivalent to a unitary one. Hence, given a compact matrix quantum group $G=(C(G),u)$, we may assume that $u$ is unitary. At the same time $\bar u=(u_{ij}^*)$ is a non-degenerate representation, so there exists an invertible matrix $F$ such that $F\bar u F^{-1}$ is unitary. Consequently, any compact matrix quantum group $G$ is, up to similarity, a quantum subgroup of the \emph{universal unitary quantum group} $U^+(F)$ \cite{VDW96} for some $F\in M_N(\C)$, $N\in\N$, whose C*-algebra is defined by
$$C(U^+(F))=C^*(u_{ij};\;i,j=1,\dots,N\mid u\text{ and }F\bar uF^{-1}\text{ are unitary}).$$

Assuming $F\bar F=c\,1_N$ for some $c\in\R$, we also define the \emph{universal orthogonal quantum group} $O^+(F)$ \cite{Ban96} by
$$C(O^+(F))=C^*(u_{ij};\;i,j=1,\dots,N\mid u=F\bar uF^{-1}\text{ is unitary}).$$

For a pair of representations $u\in M_n(C(G))$, $v\in M_m(C(G))$, we denote by
$$\Mor(u,v):=\{T\colon\C^n\to\C^m\mid Tu=vT\}$$
the space of \emph{intertwiners} between $u$ and $v$. A representation $u$ is called \emph{irreducible} if $\Mor(u,u)=\C\cdot 1$. It holds that any representation is a direct sum of irreducible ones.

For a given quantum group $G$, we denote by $\Irr G$ the set of equivalence classes of $G$. For $\alpha\in\Irr G$, we denote by $u^\alpha$ some unitary representative of the class $\alpha$. It holds that the matrix elements of all the $u^\alpha$, $\alpha\in\Irr G$ form a linear basis of $O(G)$.

\begin{rem}
Let $G$ and $H$ be compact (matrix) quantum groups such that $C(H)\subset C(G)$ and $\Delta_H=\Delta_G|_{C(H)}$. We may say that $H$ is a \emph{quotient} of $G$. Then
$$\Irr H=\{\alpha\in\Irr G\mid [u^\alpha]_{ij}\in C(H)\;\forall i,j\}\subset\Irr G.$$
Indeed, any representation of $H$ is by definition a representation of $G$. The notion of irreducibility does not change if we enlarge the algebra.
\end{rem}

\subsection{Grading on quantum group function spaces}
\label{secc.grading}
Given a compact matrix quantum group $G=(C(G),u)$, we denote by
$$I_G:=\{f\in\C\langle x_{ij},x_{ij}^*\rangle\mid f(u_{ij},u_{ij}^*)=0\}$$
the ideal determining the algebras $O(G)=\C\langle x_{ij},x_{ij}^*\rangle/I_G$ and $C(G)=C^*(O(G))$.

There is a natural structure of a $\Z_k$-grading on the algebra $\C\langle x_{ij},x_{ij}^*\rangle$ given by associating degree one to the variables $x_{ij}$, and associating degree minus one to the variables $x_{ij}^*$. In this article, by a $\Z_k$-grading we will always mean this particular grading.

If the ideal $I_G$ is homogeneous, then the $*$-algebra $O(G)$ inherits this grading. Moreover, this grading passes also to the \emph{fusion semiring} of irreducible representations in the following sense. For any $\alpha\in\Irr G$, there is $d_\alpha\in\Z_k$ such that all the matrix entries of $u^\alpha$ are $\Z_k$ homogeneous of degree $d_\alpha$. We will call $d_\alpha$ the \emph{degree} of the irreducible $u^\alpha$.

\section{Representation categories}

The main point of this section is to formulate the Tannaka--Krein duality for unitary compact matrix quantum groups. The Tannaka--Krein duality for quantum groups was first formulated by Woronowicz \cite{Wor88}. Essentially it says that any compact quantum group can be recovered from its representation category. Since then, many formulations of this statement appeared -- from very categorical ones such as in \cite{NT13} to very concrete ones such as \cite{Mal18}. We will stick here to the latter approach reformulating it a bit in the spirit of \cite{Fre17} to fit into our setting.

All the concepts presented in this section are well known to the experts. Therefore, we try to keep the section very brief. On the other hand, since the topic is quite new and rapidly developing, it is hard to give some general reference here. Some of the concrete notation and formulations of definitions and propositions are actually author's original. We refer to the author's PhD thesis \cite{Gro20} for a more detailed discussion of the concepts. See also \cite{Banbook}.

\subsection{Two-coloured categories}

Consider a compact matrix quantum group $G=(C(G),u)\subset U^+(F)$, $F\in\GL(N,\C)$. Denote $u^{\wcol}:=u$, $u^{\bcol}:=F\bar uF^{-1}$. Denote by $\W$ the free monoid over the alphabet $\{\wcol,\bcol\}$. For any word $w\in\W$, we denote $u^{\otimes w}$ the corresponding tensor product of the representations $u^{\wcol}$ and $u^{\bcol}$.

For a pair of words $w_1,w_2\in\W$, denote
\begin{align*}
\Cat_G(w_1,w_2):\!&=\Mor(u^{\otimes w_1},u^{\otimes w_2})\\&=\{T\colon(\C^N)^{\otimes |w_1|}\to(\C^N)^{\otimes |w_2|}\mid Tu^{\otimes w_1}=u^{\otimes w_2}T\}.
\end{align*}
Such a collection of vector spaces forms a rigid monoidal $*$-category in the following sense
\begin{enumerate}
\item For $T\in\Cat_G(w_1,w_2)$, $T'\in\Cat_G(w_1',w_2')$, we have $T\otimes T'\in\Cat_G(w_1w_1',w_2w_2')$.
\item For $T\in\Cat_G(w_1,w_2)$, $S\in\Cat_G(w_2,w_3)$, we have $ST\in\Cat_G(w_1,w_3)$.
\item For $T\in\Cat_G(w_1,w_2)$, we have $T^*\in\Cat_G(w_2,w_1)$
\item For every word $w\in\W$, we have $1_N^{\otimes |w|}\in\Cat_G(w,w)$.
\item There exist vectors $\xi_{\pairpart[wb]}\in\Cat_G(\emptyset,\wcol\bcol)$ and $\xi_{\pairpart[bw]}\in\Cat_G(\emptyset,\bcol\wcol)$ called the \emph{duality morphisms} such that
\begin{equation}
\label{eq.conjeq}
(\xi_{\pairpart[wb]}^*\otimes 1_{\C^N})(1_{\C^N}\otimes\xi_{\pairpart[bw]})=1_{\C^N},\qquad(\xi_{\pairpart[bw]}^*\otimes 1_{\C^N})(1_{\C^N}\otimes\xi_{\pairpart[wb]})=1_{\C^N}.
\end{equation}
\end{enumerate}

For the last point, we can write explicit formulae
\begin{equation}
\label{eq.F}
[\xi_{\pairpart[wb]}]_{ij}=F_{ji},\qquad [\xi_{\pairpart[bw]}]_{ij}=[\bar F^{-1}]_{ji}.
\end{equation}

\begin{defn}
Consider a natural number $N\in\N$. Let $\Cat(w_1,w_2)$ be a collection of vector spaces of linear maps $(\C^N)^{\otimes |w_1|}\to(\C^N)^{\otimes |w_2|}$ satisfying the conditions (1)--(5) above. Then we call $\Cat$ a \emph{two-coloured representation category}. For any collection of sets $C(w_1,w_2)$ of linear maps $(\C^N)^{\otimes |w_1|}\to(\C^N)^{\otimes |w_2|}$ satisfying (5), we denote by $\langle C\rangle$ the smallest category containing $C$. We say that $C$ \emph{generates} this category.
\end{defn}

The justification for this name is given by the following formulation of the Tannaka--Krein duality, which comes from \cite{Fre17}.

\begin{thm}[Woronowicz--Tannaka--Krein duality for CMQG]
\label{T.TK}
Let $\Cat$ be a two-coloured representation category. Then there exists a unique compact matrix quantum group $G$ such that $\Cat=\Cat_G$. This quantum group is determined by the ideal
$$I_G=\spanlin\{[Tu^{\otimes w_1}-u^{\otimes w_2}T]_{\mathbf{ji}}\mid T\in\Cat(w_1,w_2)\}.$$
\end{thm}
We give only a sketch of proof here. For a detailed explanation, see \cite[Theorem~3.4.6]{Gro18}. The proof closely follows the proof formulated for orthogonal quantum groups in \cite{Mal18}.
\begin{sproof}
In order to give a sense to the formula for $I_G$, we need to specify the matrix $F$, so that the matrix $u^{\bcol}$ is well defined. We fix the duality morphisms satisfying Eqs.~\eqref{eq.conjeq} and define $F$ according to Eqs.~\eqref{eq.F}. Note that the duality morphisms are not defined uniquely by Eqs.~\eqref{eq.conjeq}. Part of the ``uniqueness'' statement is that the resulting quantum group does not depend on the particular choice of $F$.

It is straightforward to check that $I_G$ is a biideal. We obviously have $I_G\supset I_{U^+(F)}$. Then we can check that $I_G/I_{U^+(F)}\subset O(U^+(F))$ is a Hopf $*$-ideal. Hence $I_G$ defines a compact quantum group. (See \cite{Mal18,Gro20} for details.)

Now, let $\tilde\xi_{\pairpart[wb]}$, $\tilde\xi_{\pairpart[bw]}$ be alternative solutions of \eqref{eq.conjeq} and let $\tilde F$ be the alternative matrix and $\tilde G$ the alternative resulting quantum group. Then we have
$$(\tilde\xi_{\pairpart[bw]}^*\otimes 1_{\C^N})(1_{\C^N}\otimes\xi_{\pairpart[wb]})=F\tilde F^{-1}\in\Cat_G(\bcol,\bcol)=\Mor(u^{\bcol},u^{\bcol}).$$
This means that $u^{\bcol}=\tilde FF^{-1}u^{\bcol}F\tilde F^{-1}=\tilde F\bar u\tilde F^{-1}$ and hence $G\subset\tilde G$. From symmetry, we have $G=\tilde G$.

It remains to prove that we indeed have $\Cat_G=\Cat$ (from construction, we can actually easily see that $\Cat_G\supset\Cat$) and that $G$ is a unique quantum group with this property, that is, if $\Cat_{\tilde G}=\Cat$ for some quantum group $\tilde G\subset U^+(F)$, then surely $G=\tilde G$ (again, from construction, we obviously have $G\supset\tilde G$). This can be proven using the double commutant theorem, see \cite{Mal18,Gro20}.
\end{sproof}


\begin{prop}
Let $G\subset U^+(F)$ be a compact matrix quantum group. Suppose that the associated category $\Cat_G$ is generated by some collection $C(w_1,w_2)$. Then $I_G\subset\C\langle x_{ij},x_{ij}^*\rangle$ as an ideal is generated by
$$\{[Tx^{\otimes w_1}-x^{\otimes w_2}T]_{\mathbf{ji}}\mid T\in C(w_1,w_2)\}.$$
\end{prop}
\begin{proof}
Denote by $I$ the ideal generated by $C$ as formulated above. Obviously, we have $I\subset I_G$. To prove the opposite inclusion, it is enough to prove that
$$\Cat(w_1,w_2):=\{T\colon (\C^N)^{\otimes |w_1|}\to(\C^N)^{\otimes |w_2|}\mid Tx^{\otimes w_1}-x^{\otimes w_2}T\in I\}$$
form a category. Then, since obviously $C\subset\Cat$ and hence $\Cat_G\subset\Cat$, we must have $I_G\subset I$.

So, denote $A:=\C\langle x_{ij},x_{ij}^*\rangle/I$ and by $v_{ij}$ denote the images of $x_{ij}$ by the natural homomorphism. Taking $T_1\in\Cat(w_1,w_2)$, $T_2\in\Cat(w_2,w_3)$. Then
$$T_2T_1v^{\otimes w_1}=T_2v^{\otimes w_2}T_1=v^{\otimes w_3}T_2T_1,$$
so $T_2T_1\in\Cat(w_1,w_3)$. For tensor product and involution, the proof is similar.
\end{proof}

\begin{rem}
\label{R.OGcat}
Recall the universal orthogonal quantum group $O^+(F)\subset U^+(F)$, which is defined by the relation $u=F\bar uF^{-1}$. The relation can also be written as $u^{\wcol}=u^{\bcol}$. Consequently, we have $u^{\otimes w}=u^{\otimes |w|}$ for any $w\in\W$, so only the length of the word $w$ matters. For any $G\subset O^+(F)$, we define
$$\Cat_G(k,l):=\Mor(u^{\otimes k},u^{\otimes l})=\Cat_G(w_1,w_2),$$
where $w_1,w_2$ are any words with $|w_1|=k$ and $|w_2|=l$.
\end{rem}

\subsection{Representation categories of quantum subgroups, intersections, and topological generation}

In this section, we would like to briefly explain the concepts of the quantum group intersection and topological generation. Together with quantum subgroups, we relate those notions with the associated ideals of algebraic relations $I_G$ and the representation categories $\Cat_G$. See also \cite[Sects. 2.3.4, 2.5.4, 2.5.5, 3.4.5]{Gro20} for more detailed discussion.

\begin{prop}
\label{P.subgr}
Consider $G,H\subset U^+(F)$. Then the following are equivalent.
\begin{enumerate}
\item $H\subset G$,
\item $I_H\supset I_G$,
\item $\Cat_H(w_1,w_2)\supset\Cat_G(w_1,w_2)$ for all $w_1,w_2\in\W$.
\end{enumerate}
\end{prop}
\begin{proof}
The equivalence $(1)\Leftrightarrow (2)$ follows directly from the definition of quantum subgroup. The equivalence $(2)\Leftrightarrow (3)$ follows from Tannaka--Krein duality (Thm.~\ref{T.TK}).
\end{proof}

Consider $H_1,H_2\subset U^+(F)$. Their \emph{intersection} $H_1\cap H_2$ is the largest compact matrix quantum group contained in both $H_1$ and $H_2$. This notion was recently heavily used especially in the work of Teodor Banica. See \cite{Banbook}.

\begin{prop}
Consider $G,H_1,H_2\subset U^+(F)$. Then the following are equivalent.
\begin{enumerate}
\item $G=H_1\cap H_2$,
\item $I_G=I_{H_1}+I_{H_2}$,
\item $\Cat_G=\langle\Cat_{H_1},\Cat_{H_2}\rangle$.
\end{enumerate}
\end{prop}
\begin{proof}
The equivalence follows from Proposition~\ref{P.subgr}. $G$ being the largest quantum group contained in $H_1$ and $H_2$ is equivalent to $I_G$ being the smallest ideal containing $I_{H_1}$ and $I_{H_2}$, which is equivalent to $\Cat_G$ being the smallest category containing $\Cat_{H_1}$ and $\Cat_{H_2}$.
\end{proof}

Consider $H_1,H_2\subset U^+(F)$. The smallest quantum group $G$ containing both $H_1$ and $H_2$ is said to be \emph{topologically generated} by $H_1$ and $H_2$. We denote it by $G=\langle H_1,H_2\rangle$. This notion goes back to \cite{Chi15,BCV17}.

\begin{prop}
Consider $G,H_1,H_2\subset U^+(F)$. Then the following are equivalent.
\begin{enumerate}
\item $G=\langle H_1, H_2\rangle$,
\item $I_G$ is the largest Hopf $*$-ideal in $I_{H_1}\cap I_{H_2}$,
\item $\Cat_G(w_1,w_2)=\Cat_{H_1}(w_1,w_2)\cap\Cat_{H_2}(w_1,w_2)$ for all $w_1,w_2\in\W$.
\end{enumerate}
\end{prop}
\begin{proof}
Again, the equivalence follows from Proposition~\ref{P.subgr}. $G$ being the smallest quantum group containing $H_1$ and $H_2$ is equivalent to $I_G$ being the largest ideal contained in $I_{H_1}$ and $I_{H_2}$, which is equivalent to $\Cat_G$ being the largest category contained in $\Cat_{H_1}$ and $\Cat_{H_2}$.
\end{proof}

\subsection{Frobenius reciprocity}
\label{secc.Frobenius}
We define an involution on the set of two colours $\{\wcol,\bcol\}$ mapping $\wcol\mapsto\bcol$, $\bcol\mapsto\wcol$. We extend this operation as a~homomorphism on the monoid $\W$, denote it by bar $w\mapsto\bar w$, and call it the \emph{colour inversion}. We denote by $w^*$ the colour inversion composed with reflection on $w$ (that is, reading the word backwards).

Consider a two-coloured representation category $\Cat$ and fix the duality morphisms $\xi_{\pairpart[wb]}$, $\xi_{\pairpart[bw]}$. We define a map $\Rrot\colon\Cat(w_1,w_2a)\to\Cat(w_1\bar a,w_2)$ for $a\in\{\wcol,\bcol\}$ by $T\mapsto (1_{(\C^N)^{|w_2|}}\otimes\xi^*)(T\otimes 1_{\C^N})$ with $\xi=\xi_{\pairpart[wb]}$ if $a=\wcol$ and $\xi=\xi_{\pairpart[bw]}$ if $a=\bcol$. We call this map the \emph{right rotation}.

This map has an inverse $\Rrot^{-1}\colon\Cat(w_1a,w_2)\to\Cat(w_1,w_2\bar a)$ given by $T\mapsto (T\otimes 1_{\C^N})(1_{(\C^N)^{\otimes |w_1|}}\otimes\xi)$. Similarly, we can define the left rotation $\Lrot\colon \Cat(aw_1,w_2)\to\Cat(w_1,\bar aw_2)$. As a consequence, we have the following.

\begin{prop}
Let $\Cat$ be a two-coloured representation category. Then $\Cat$ is generated by the collection $\Cat(\emptyset,w)$ with $w$ running through $\W$.
\end{prop}
\begin{proof}
We have $\Cat(w_1,w_2)=\Rrot^{|w_1|}\Cat(\emptyset,w_1w_2^*)$.
\end{proof}

\begin{cor}
\label{C.idealgen}
Let $G$ be a compact quantum group. Then the ideal $I_G$ is generated by relations of the form $u^{\otimes w}\xi=\xi$.
\end{cor}

Some authors denote $\mathop{\rm Fix}(u^{\otimes w}):=\Mor(1,u^{\otimes w})=\Cat_G(\emptyset,w)$. Such particular intertwiners $\xi\in\Mor(1,u^{\otimes w})$ satisfying $u^{\otimes w}\xi=\xi$ are called the \emph{fixed points} of $u^{\otimes w}$.

Not only that the collection $\Cat(\emptyset,w)$ already determines the whole category $\Cat$. We can even characterize representation categories in terms of those fixed point spaces by introducing some alternative operations. The following essentially reformulates the operations defined in \cite{Gro18}.

\begin{defn}
\label{D.Frobenius}
Consider a~two-coloured representation category $\Cat$. Denote for simplicity $\xi_{\wcol}:=\xi_{\pairpart[wb]}$ and $\xi_{\bcol}:=\xi_{\pairpart[bw]}$. We define the following operations on the sets $\Cat(\emptyset,w)$.
\begin{itemize}
\item If $a_i$ and $a_{i+1}$ have opposite colours, we define the \emph{contraction}:
\begin{gather*}
\Pi_i\colon\Cat(\emptyset,a_1\cdots a_k)\to\Cat(\emptyset,a_1\cdots a_{i-1}a_{i+2}\cdots a_k),\\
\Pi_i\eta:=(1_N\otimes\cdots\otimes 1_N\otimes \xi_{a_i}^*\otimes 1_N\otimes \cdots\otimes 1_N)\eta.
\end{gather*}
Pictorially,
$$\Pi_i\eta=
\BigPartition{
\Pline(0.5,1)(8.5,1)
\Pline(0.5,0.5)(8.5,0.5)
\Pline(0.5,0.5)(0.5,1)
\Pline(8.5,0.5)(8.5,1)
\Ptext(4.5,0.75){$\eta$}
\Pline(1,0.5)(1,0.2)
\Ptext(2,0.3){$\cdots$}
\Pline(3,0.5)(3,0.2)
\Pblock 0.5to0.3:4,5
\Ptext(4.5,0.1){$\scriptstyle\xi_{a_i}^*$}
\Pline(6,0.5)(6,0.2)
\Ptext(7,0.3){$\cdots$}
\Pline(8,0.5)(8,0.2)}.$$
On elementary tensors, it acts as
$$\Pi_i(\eta_1\otimes\cdots\otimes\eta_k)=(\eta_{i+1}^tF\eta_i)\,\eta_1\otimes\cdots\otimes\eta_{i-1}\otimes\eta_{i+2}\otimes\cdots\otimes\eta_k.$$
\item We define the \emph{rotation}:
\begin{gather*}
R\colon\Cat(\emptyset,a_1\cdots a_k)\to\Cat(\emptyset,a_ka_1\cdots a_{k-1}),\qquad R:=\Lrot\circ\Rrot,\quad\hbox{so}\\
R\eta=(1_N\otimes\cdots\otimes 1_N\otimes\xi_{a_k}^*)(1_N\otimes\eta\otimes 1_N)\xi_{a_k}.
\end{gather*}
Pictorially,
$$R\eta=
\BigPartition{
\Pline(0.5,0.75)(6.5,0.75)
\Pline(0.5,0.25)(6.5,0.25)
\Pline(0.5,0.25)(0.5,0.75)
\Pline(6.5,0.25)(6.5,0.75)
\Ptext(3.5,0.5){$\eta$}
\Pline(1,0.25)(1,-0.1)
\Pline(2,0.25)(2,-0.1)
\Ptext(3,0){$\cdots$}
\Pline(4,0.25)(4,-0.1)
\Pline(5,0.25)(5,-0.1)
\Pblock 0.25to0.1:6,7
\Pblock 0.25to1:0,7
\Pline(0,0.25)(0,-0.1)
\Ptext(6.5,-0.1){$\scriptstyle\xi_{a_k}^*$}
\Ptext(3.5,1.2){$\scriptstyle\xi_{a_k}$}
}.$$
On elementary tensors, it acts as
$$R(\eta_1\otimes\cdots\otimes\eta_k)=(F\bar F\eta_k)\otimes\eta_1\otimes\cdots\otimes\eta_{k-1}.$$
Note that the rotation is obviously invertible with $R^{-1}=\Rrot^{-1}\circ\Lrot^{-1}\colon\Cat(\emptyset,a_1\cdots a_k)\to\Cat(\emptyset,a_2\cdots a_ka_1)$.
\item We define the \emph{reflection}:
\begin{gather*}
\star\colon\Cat(\emptyset,a_1\cdots a_k)\to\Cat(\emptyset,a_k\cdots a_1)\\
\eta^\star:=\Rrot^{-k}\eta^*=(\eta^*\otimes 1_{N_{a_k}}\otimes\cdots\otimes 1_{N_{a_1}})\xi_{a_1\cdots a_k},
\end{gather*}
where $\xi_{a_1\cdots a_k}$ is the duality morphism associated to the object $a_1\cdots a_k$.
Pictorially,
$$R\eta=
\BigPartition{
\Pline(0.5,0.5)(5.5,0.5)
\Pline(0.5,0)(5.5,0)
\Pline(0.5,0)(0.5,0.5)
\Pline(5.5,0)(5.5,0.5)
\Ptext(3,0.25){$\eta^*$}
\Pblock 0.5to0.6:5,6
\Pblock 0.5to0.7:4,7
\Pblock 0.5to0.9:2,9
\Pblock 0.5to1:1,10
\Pline(6,0.5)(6,0)
\Pline(7,0.5)(7,0)
\Ptext(3,0.6){$\cdots$}
\Ptext(8,0.3){$\cdots$}
\Pline(9,0.5)(9,0)
\Pline(10,0.5)(10,0)
}.$$
On elementary tensors, it acts as
$$(\eta_1\otimes\cdots\otimes\eta_k)^\star=(F^{a_k}\bar\eta_k)\otimes\cdots\otimes(F^{a_1}\bar\eta_1).$$
\end{itemize}
\end{defn}

\begin{prop}
\label{P.Frobenius}
For any two-coloured representation category $\Cat$, the collection of sets $\Cat(\emptyset,w)$, $w\in\W$ is closed under tensor products, contractions, rotations, inverse rotations, and reflections. Conversely, for any collection of vector spaces $\Cat(w)\subset(\C^N)^{\otimes |w|}$ that is closed under tensor products, contractions, rotations, inverse rotations, and reflections and satisfies axiom~(5) of two-coloured representation categories, the sets
$$\Cat(w_1,w_2):=\{\Rrot^{|w_1|}\xi\mid \xi\in\Cat(w_2w_1^*)\}=\{\Lrot^{-|w_1|}p\mid p\in\Cat(w_1^*w_2)\}$$
form a~two-coloured representation category.
\end{prop}
\begin{proof}
The first part of the proposition follows from the fact that all the new operations are defined using the category operations of tensor product, composition, and involution.

The converse statement is proved by expressing the category operations in terms of the new operations:
\begin{align*}
\Lrot^{-k}\xi\otimes\Rrot^{k'}\eta&=\Lrot^{-k}\Rrot^{k'}(\xi\otimes\eta),\\
(\Rrot^k\xi)^*&=\Rrot^l\xi^\star,\\
(\Rrot^l\eta)(\Rrot^k\xi)&=\Rrot^k\Pi_{m+1}\Pi_{m+2}\cdots\Pi_{m+l}(\eta\otimes\xi),
\end{align*}
where we assume that $\xi\in\Cat(w_1w_2)$ and $\eta\in\Cat(w_1'w_2')$ (for the first row), resp.\ $\eta\in\Cat(w_2w_3)$ (for the last row).
\end{proof}

\subsection{Free and tensor product}
\label{secc.prods}

The following two constructions were defined by Wang.

\begin{prop}[\cite{Wan95free}]
Let $H_1=(C(H_1),v_1)$ and $H_2=(C(H_2),v_2)$ be compact matrix quantum groups. Then $H_1*H_2:=(C(H_1)*_\C C(H_2),v_1\oplus v_2)$ is a compact matrix quantum group. For the co-multiplication we have that
$$\Delta_*([v_1]_{ij})=\Delta_{H_1}([v_1]_{ij}),\quad\Delta_*([v_2]_{kl})=\Delta_{H_2}([v_2]_{kl}).$$
\end{prop}

\begin{prop}[\cite{Wan95tensor}]
Let $H_1=(C(H_1),v_1)$ and $H_2=(C(H_2),v_2)$ be compact matrix quantum groups. Then $H_1\times H_2:=(C(H_1)\otimes_{\rm max}C(H_2),v_1\oplus v_2)$ is a compact matrix quantum group. For the co-multiplication we have that
$$\Delta_\times([v_1]_{ij}\otimes 1)=\Delta_{H_1}([v_1]_{ij}),\quad\Delta_\times(1\otimes [v_2]_{kl})=\Delta_{H_2}([v_2]_{kl}).$$
\end{prop}

The tensor product of quantum groups is a generalization of the group direct product. The free product should also be seen as some free version of the direct product. Note in particular that the free product of compact quantum groups does not generalize the group free product since the freeness occurs in the C*-algebra multiplication not in the quantum group comultiplication. For this reasons, some authors call it rather \emph{dual free product}.

We will focus here mainly on products of quantum groups $H=(C(H),v)$ with cyclic group duals $\hat\Z_k$. Denote by $z$ the generator of the C*-algebra $C^*(\Z_k)$, so the fundamental representation of $H*\hat\Z_k$ or $H\times\hat\Z_k$ is of the form $u:=v\oplus z$. Now, we would like to describe the corresponding representation category. For our purposes, it will be convenient not to restrict only to the intertwiner spaces between tensor products of $u^{\wcol}$ and $u^{\bcol}$, but to keep track of the blocks $v$ and $z$.

Consider $G\subset U^+(F)*\Z_k$ for some $k\in\N_0$ with fundamental representation of the form $u=v\oplus z$, where $z$ is one-dimensional. Denote by $N$ the size of $v$. We define a~monoid $\W_k$ with generators $\qcol,\Qcol,\tcol,\Tcol$ and relations $\tcol\Tcol=\Tcol\tcol=\emptyset$, $\tcol^k=\emptyset$, where $\emptyset$ is the monoid identity. We denote $u^{\qcol}:=v^{\wcol}$, $u^{\Qcol}:=v^{\bcol}$, $u^{\tcol}:=z$, $u^{\Tcol}=z^*$, so $u^{\otimes w}$ is the tensor product of the corresponding representations. We denote by $[w]$ the number of white and black squares in a given word $w$ (which is well defined in contrast with the overall length of $w$).

Then we can associate to $G$ the following category
\begin{align*}
\Cat_G(w_1,w_2):\!&=\Mor(u^{\otimes w_1},u^{\otimes w_2})\\&=\{T\colon(\C^N)^{\otimes [w_1]}\to(\C^N)^{\otimes [w_2]}\mid Tu^{\otimes w_1}=u^{\otimes w_2}T\},
\end{align*}
where $w_1,w_2\in\W_k$.

The axiomatization of a two-coloured representation category extends to this case as follows.

\begin{defn}
Consider a natural number $N\in\N$ and $k\in\N_0$. A~collection $\Cat(w_1,w_2)$, $w_1,w_2\in\W_k$ of vector spaces of linear maps $(\C^N)^{\otimes [w_1]}\to(\C^N)^{\otimes [w_2]}$ is called a~\emph{$\Z_k$-extended representation category} if it satisfies the following
\begin{enumerate}
\item For $T\in\Cat_G(w_1,w_2)$, $T'\in\Cat_G(w_1',w_2')$, we have $T\otimes T'\in\Cat_G(w_1w_1',w_2w_2')$.
\item For $T\in\Cat_G(w_1,w_2)$, $S\in\Cat_G(w_2,w_3)$, we have $ST\in\Cat_G(w_1,w_3)$.
\item For $T\in\Cat_G(w_1,w_2)$, we have $T^*\in\Cat_G(w_2,w_1)$
\item For every word $w\in\W_k$, we have $1_N^{\otimes [w]}\in\Cat_G(w,w)$.
\item There exist vectors $\xi_{\pairpart[qQ]}\in\Cat_G(\emptyset,\qcol\Qcol)$ and $\xi_{\pairpart[Qq]}\in\Cat_G(\emptyset,\Qcol\qcol)$ called the \emph{duality morphisms} such that
\begin{equation}
\label{eq.conjeq}
(\xi_{\pairpart[qQ]}^*\otimes 1_{\C^N})(1_{\C^N}\otimes\xi_{\pairpart[Qq]})=1_{\C^N},\qquad(\xi_{\pairpart[Qq]}^*\otimes 1_{\C^N})(1_{\C^N}\otimes\xi_{\pairpart[qQ]})=1_{\C^N}.
\end{equation}
\end{enumerate}
\end{defn}

In other words, $\Z_k$-extended categories are rigid monoidal $*$-categories with $\W_k$ being the monoid of objects and morphisms realized by linear maps. Note that it is not necessary to assume the existence of duality morphisms for the triangles $\tcol$ since we automatically have $\Cat(\emptyset,\tcol\Tcol)=\Cat(\emptyset,\Tcol\tcol)=\Cat(\emptyset,\emptyset)\owns 1$ being the duality morphism.

We can reformulate many results mentioned above to this $\Z_k$-extended case. In particular, the Tannaka--Krein duality associates a~compact matrix quantum group $G\subset U^+(F)*\hat\Z_k$ to any $\Z_k$-extended representation category.

Also the Frobenius reciprocity holds for $\Z_k$-extended representation categories and the operations on the fixed point spaces can be defined in a~similar way. In particular, the contraction $\Cat(\emptyset,w_1\qcol\Qcol w_2)\to\Cat(\emptyset,w_1w_2)$ or $\Cat(\emptyset,w_1\Qcol\qcol w_2)\to\Cat(\emptyset,w_1w_2)$ is defined the same way as in Sect.~\ref{secc.Frobenius}. On the other hand the contraction $\Cat(\emptyset,w_1\tcol\Tcol w_2)\to\Cat(\emptyset,w_1w_2)$ is simply the identity since already on the level of objects we have $w_1\tcol\Tcol w_2=w_1w_2$. Similarly, the rotation of squares is defined the same way as the rotation of circles in Sect.~\ref{secc.Frobenius} and will be denoted by $R$. The rotation of triangles is then simply the identity. Consequently, we have $R^{[w_2]}\colon\Cat(\emptyset,w_1w_2)\to\Cat(\emptyset,w_2w_1)$.


A way how to model $\Z_2$-extended representation categories using partitions is described in \cite{GW19}. Many new examples of quantum groups were obtained using this approach. In particular, several new product constructions interpolating the free and tensor product.

\section{Glued products}

In \cite{TW17,Gro18,GW19}, the representation categories of glued products of orthogonal easy quantum groups with cyclic groups were studied. Here, we revisit the theory dropping the easiness assumption.

\subsection{Gluing procedure}

\begin{defn}
\label{D.gluedver}
Let $G$ be a compact matrix quantum group with fundamental representation of the form $u=v_1\oplus v_2$. Denote by $\tilde A$ the C*-subalgebra of $C(G)$ generated by elements of the form $[v_1]_{ij}[v_2]_{kl}$. Then $\tilde G:=(\tilde A,v_1\otimes v_2)$ is a compact matrix quantum group called the \emph{glued version} of $G$.
\end{defn}

A particular example of the gluing procedure is the \emph{glued free product} $H_1\tilstar H_2$ defined as the glued version of the free product $H_1*H_2$ and the \emph{glued tensor product} $H_1\tiltimes H_2$ defined as the glued version of the tensor product $H_1\times H_2$. Those two definitions were first formulated in \cite{TW17}.

Again, we will be interested in particular in the case when $u=v\oplus z$, where $z$ is one-dimensional. The glued free product $H\tilstar\hat\Z_k$ is also called the \emph{free $k$-complexification} of $H$ and the glued tensor product $H\tiltimes\hat\Z_k$ is called the \emph{tensor $k$-complexification}. The free complexification was studied already by Banica in \cite{Ban99,Ban08}.

\begin{rem}
The glued version $\tilde G$ of a quantum group $G$ is by definition a quotient of $G$. It may happen that the elements $[v_1]_{ij}[v_2]_{kl}$ already generate the whole C*-algebra $C(G)$, so $C(\tilde G)=C(G)$. In this case, we have that $\tilde G$ is isomorphic to $G$. However, the $\tilde G$ and $G$ are still distinct as compact \emph{matrix} quantum groups since their fundamental representations are different.

The same holds in particular for the glued products and for the complexifications. Considering $G=(C(G),v)$ and $\hat\Z_k=(z,C^*(\Z_k))$, we have that $G\tiltimes\hat\Z_k\simeq G\times\hat\Z_k$ or $G\tilstar\hat\Z_k\simeq G*\hat\Z_k$ if the elements $v_{ij}z$ actually generate the whole algebra $C(G)\otimes_{\rm max}C^*(\Z_k)$, resp.\ $C(G)*_{\C}C^*(\Z_k)$.
\end{rem}

Now, we are going to characterize the representation categories of the glued versions.

\begin{defn}
\label{D.gluew}
Let us fix $k\in\N_0$. Then for any word $w\in\W$ we associate its \emph{glued version} $\tilde w\in\W_k$ mapping $\wcol\mapsto\qcol\tcol$, $\bcol\mapsto\Tcol\Qcol$.
\end{defn}

\begin{prop}
\label{P.repglue}
Consider $G\subset U^+(F)*\hat\Z_k$ with fundamental representation $u=v\oplus z$. Let $\Cat_G$ be the associated $\Z_k$-extended representation category. Let $\tilde G$ be the glued version of $G$. Then
$$\Cat_{\tilde G}(w_1,w_2)=\Cat_G(\tilde w_1,\tilde w_2)$$
for every $w_1,w_2\in\W$. That is $\Cat_{\tilde G}$ is a~full subcategory of $\Cat_G$ given by considering the glued words only. The ideal associated to $\tilde G$ can be described as
$$I_{\tilde G}=\{f\in\C\langle \tilde x_{ij},\tilde x_{ij}^*\rangle\mid f(x_{ij}z,z^*x_{ij}^*)\in I_G\}\simeq I_G\cap\C\langle x_{ij}z,z^*x_{ij}^*\rangle.$$
\end{prop}
\begin{proof}
Denote by $\tilde v=vz$ the fundamental representation of~$\tilde G$. Consider a~word $w$ and its glued version $\tilde w$. Directly from the definitions of $\tilde v$ and $\tilde w$, we have $\tilde v^{\otimes w}=u^{\otimes \tilde w}$. So,
$$\Cat_{\tilde G}(w_1,w_2)=\Mor(\tilde v^{\otimes w_1},\tilde v^{\otimes w_2})=\Mor(u^{\otimes\tilde w_1},u^{\otimes\tilde w_2})=\Cat_G(\tilde w_1,\tilde w_2).$$
For the ideal, we have
\begin{align*}
I_{\tilde G}&=\{f\in\C\langle\tilde x_{ij},\tilde x_{ij}^*\rangle\mid 0=f(\tilde v_{ij},\tilde v_{ij}^*)=f(v_{ij}z,z^*v_{ij}^*)\text{ in $C(\tilde G)\subset C(G)$}\}\\&=\{f\in\C\langle\tilde x_{ij},\tilde x_{ij}^*\rangle\mid f(x_{ij}z,z^*x_{ij}^*)\in I_G\}.\qedhere
\end{align*}

\end{proof}


\subsection{Glued version and projective version}
\label{secc.proj}

As a side remark, let us relate the new notion of \emph{glued version} with already existing notion of \emph{projective version}.

\begin{defn}[\cite{BV10}]
Let $G\subset U^+(F)$ be a compact matrix quantum group. We define its \emph{projective version} as $PG:=(C(PG),u^{\wcol}\otimes u^{\bcol})$, where $C(PG)$ is the C*-subalgebra of $C(G)$ generated by the elements $u_{ij}^{\wcol}u_{kl}^{\bcol}$.
\end{defn}

\begin{prop}
Consider a compact matrix quantum group $G$ with fundamental representation of the form $v_1\oplus v_2$. Denote by $G':=(C(G),v_1^{\wcol}\oplus v_2^{\bcol})\simeq G$. Let $\tilde G$ be the glued version of $G$. Then $\tilde G\simeq PG'$.
\end{prop}
\begin{proof}
By definition, we have that $PG'$ is determined by a C*-subalgebra $C(PG')\subset C(G)$ generated by matrix elements of the fundamental representation of $PG'$, which is of the form $(v_1^{\wcol}\otimes v_1^{\bcol})\oplus(v_1^{\wcol}\otimes v_2^{\wcol})\oplus(v_2^{\bcol}\otimes v_1^{\bcol})\oplus(v_2^{\bcol}\otimes v_2^{\wcol})$. In contrast, $C(\tilde G)\subset C(G)$ is generated only by $v_1^{\wcol}\otimes v_2^{\wcol}$. We need to show that $C(\tilde G)$ and $C(PG')$ coincide as subalgebras of $C(G)$. We can express
\begin{align*}
v_2^{\bcol}\otimes v_1^{\bcol}&=(v_1\otimes v_2)^{\bcol},\\
v_1^{\wcol}\otimes v_1^{\bcol}&=(\id\otimes\xi_{\pairpart[wb]}^{(2)*}\otimes\id)(v_1^{\wcol}\otimes v_2^{\wcol}\otimes v_2^{\bcol}\otimes v_1^{\bcol})(\id\otimes\xi_{\pairpart[wb]}^{(2)}\otimes\id)/\norm{\xi_{\pairpart[wb]}^{(2)}}^2,\\
v_2^{\bcol}\otimes v_2^{\wcol}&=(\id\otimes\xi_{\pairpart[bw]}^{(1)*}\otimes\id)(v_2^{\bcol}\otimes v_1^{\bcol}\otimes v_1^{\wcol}\otimes v_2^{\wcol})(\id\otimes\xi_{\pairpart[bw]}^{(1)}\otimes\id)/\norm{\xi_{\pairpart[bw]}^{(1)}}^2,
\end{align*}
where $\xi_{\pairpart[wb]}=\xi_{\pairpart[wb]}^{(1)}\oplus\xi_{\pairpart[wb]}^{(2)}$ and $\xi_{\pairpart[bw]}=\xi_{\pairpart[bw]}^{(1)}\oplus\xi_{\pairpart[bw]}^{(2)}$ are the duality morphisms corresponding to $v_1\oplus v_2$.
\end{proof}

\begin{rem}
Often it happens that $G$ and $G'$ are identical as compact matrix quantum groups. For instance, in the case of tensor product or free product with $\hat\Z_k$. So, we can write
$$H\tiltimes\hat\Z_k\simeq P(H\times\hat\Z_k),\qquad H\tilstar\hat\Z_k\simeq P(H*\hat\Z_k).$$
\end{rem}

\subsection{Degree of reflection}
\label{secc.degref}

We characterize the notion of the degree of reflection, which was introduced in \cite{TW18} in the categorical language and in \cite{GW19} for general compact matrix quantum groups.

\begin{defn}[\cite{TW18}]
For a word $w\in\W$, we define $c(w)$ to be the number of white circles $\wcol$ in $w$ minus the number of black circles $\bcol$ in $w$.
\end{defn}

Recall that given a quantum group $G=(C(G),u)$, we can construct a quantum subgroup of $G$ -- so called \emph{diagonal subgroup} -- imposing the relation $u_{ij}=0$ for all $i\neq j$. If we, in addition, impose the relation $u_{ii}=u_{jj}$ for all $i$ and $j$, we get a quantum group corresponding to a C*-algebra generated by a single unitary. Therefore, it must be a dual of some cyclic group.

\begin{defn}[\cite{GW19}]
Let $G$ be a quantum group and denote by $\hat\Gamma$ the quantum subgroup of $G$ given by $u_{ij}=0$, $u_{ii}=u_{jj}$ for all $i\neq j$. The order of the cyclic group $\Gamma$ is called the \emph{degree of reflection} of $G$. If the order is infinite, we set the degree of reflection to zero.
\end{defn}


\begin{lem}
\label{L.degref}
Let $G=(C(G),u)$ be a compact matrix quantum group, $k\in\N_0$. The following are equivalent.
\begin{enumerate}\setcounter{enumi}{-1}
\item The number $k$ divides the degree of reflection of $G$.
\item The mapping $u_{ij}\mapsto \delta_{ij} z$ extends to a $*$-homomorphism $\phi\colon C(G)\to C^*(\hat\Z_k)$.
\item For any $w\in\W$, $\Mor(1,u^{\otimes w})\neq\{0\}$ only if $c(w)$ is a multiple of $k$.
\item The ideal $I_G$ is $\Z_k$-homogeneous.
\item We have $G=G\tiltimes\hat\Z_k$.
\end{enumerate}
\end{lem}
\begin{proof}
$(0)\Leftrightarrow(1)$: If $k_0$ is the degree of reflection of $G$, then directly by the definition there is a~$*$\hbox{-}homomorphism $\phi\colon C(G)\to C^*(\hat\Z_{k_0})$. Such an homomorphism obviously exists also if $k$ is a divisor of $k_0$. By definition, $\hat\Z_{k_0}$ is the largest group with this property, so $k$ must be a~divisor of $k_0$.

$(1)\Rightarrow (2)$: Take any $\xi\in\Mor(1,u^{\otimes w})$, so $u^{\otimes w}\xi=\xi$. Applying the homomorphism~$\phi$, we get $z^{c(w)}\xi=\xi$. If $\xi\neq 0$, we must have $z^{c(w)}=1$, so $c(w)$ is a~multiple of~$k$.

$(2)\Rightarrow (3)$: By Corollary~\ref{C.idealgen}, $I_G$~is generated by the relations $u^{\otimes w}\xi=\xi$, $\xi\in\Mor(1,u^{\otimes w})$. Since the entries of $u^{\otimes w}$ are monomials of degree $c(w)$, the relations $u^{\otimes w}\xi=\xi$ are $\Z_{c(w)}$-homo\-ge\-neous (of degree zero). Consequently, they are also $\Z_k$-homogeneous and hence generate a~$\Z_k$-homogeneous ideal.

$(3)\Rightarrow(4)$: We need to show that $u_{ij}\mapsto u_{ij}z$ extends to a~$*$\hbox{-}isomorphism $C(G)\to C(G\tiltimes\hat\Z_k)$. To prove that it extends to a~homomorphism, take any $f\in I_G$. Suppose $f$ is $\Z_k$-homogeneous of degree~$l$. Then, since $u_{ij}$ and $z$ commute, we have $f(u_{ij}z)=f(u_{ij})z^l=0$. It is surjective directly from definition. For injectivity, note that the projection to the first tensor component $C(G)\otimes_{\rm max}C^*(\Z_k)\to C(G)$ restricts to the inverse of~$\alpha$.

$(4)\Rightarrow (1)$: We define $\phi:=(\epsilon\otimes\id)\alpha$, where $\epsilon$ is the counit of~$G$ and $\alpha$~is the $*$\hbox{-}homomorphism $C(G)\to C(G)\otimes C^*(\hat\Z_k)$. Then indeed $\phi(u_{ij})=\epsilon(u_{ij})z=\delta_{ij}z$.
\end{proof}

As a~consequence, we have the following four equivalent characterizations of the degree of reflection.

\begin{prop}
\label{P.degref}
Let $G$ be a compact matrix quantum group, $k\in\N_0$. The following are equivalent.
\begin{enumerate}
\item The number $k$ is the degree of reflection of $G$.
\item We have $\{c(w_2)-c(w_1)\mid \Cat_G(w_1,w_2)\neq\{0\}\}=k\Z$.
\item The number $k$ is the largest such that $I_G$ is $\Z_k$-homogeneous.
\item The number $k$ is the largest such that $G=G\tiltimes\hat\Z_k$.
\end{enumerate}
In items (3) and (4), we consider zero to be larger than every natural number (equivalently, consider the order defined by ``is a multiple of'').
\end{prop}
\begin{proof}
We just take the maximal~$k$ (in the above mentioned sense) satisfying the equivalent conditions in Lemma~\ref{L.degref}. For~(2) note that the set $\{c(w_2)-c(w_1)\mid\Cat_G(w_1,w_2)\neq\{0\}\}$ is indeed a~subgroup of~$\Z$. The fact that $\{c(w_2)-c(w_1)\}$ is closed under addition follows from $\Cat_G$ being closed under the tensor product. The fact that $\{c(w_2)-c(w_1)\}$ is closed under subtraction follows from $\Cat_G$ being closed under the involution. The statement~(2) in Lemma~\ref{L.degref} can be formulated as $\{c(w_2)-c(w_1)\mid \Cat_G(w_1,w_2)\neq\{0\}\}\subset k\Z$. Taking the maximal~$k$, we gain the equality.
\end{proof}

\begin{prop}
\label{P.degrefOG}
Consider $G\subset O^+(F)$. Then one of the following is true.
\begin{enumerate}
\item The degree of reflection of $G$ is one and $\Cat_G(0,k)\neq\{0\}$ for some odd $k\in\N_0$.
\item The degree of reflection of $G$ is two and $\Cat_G(k,l)=\{0\}$ for every $k+l$ odd.
\end{enumerate}
\end{prop}
\begin{proof}
Recall from Remark~\ref{R.OGcat} that the intertwiner spaces $\Cat_G(w_1,w_2)$ depend only on the length of the words $w_1$ and~$w_2$ for $G\subset O^+(F)$. This allowed us to introduce the notation $\Cat_G(k,l)$. Now the proposition follows from Proposition~\ref{P.degref}. First of all, the degree of reflection must be a divisor of two (that is, either one or two) since we have $\Cat_G(\emptyset,\wcol\wcol)=\Cat_G(\emptyset,\wcol\bcol)\owns\xi_{\pairpart[wb]}\neq 0$. Then $G$ has degree of reflection one if and only if $\Cat_G(\emptyset,w)=\Cat_G(0,|w|)\neq\{0\}$ for some word $w$ with $c(w)=1$. Such a~word with $c(w)=1$ must be of odd length.
\end{proof}

\subsection{Global colourization}
\label{secc.globcol}

The notion of globally-colourized categories was introduced in \cite{TW18} and studied in more detail in \cite{Gro18}. Here, we reformulate the results in the non-easy case.

\begin{defn}
\label{D.globcol}
A compact matrix quantum group $G=(C(G),u)$ is called \emph{globally colourized} if the following holds in $C(G)$
\begin{equation}
\label{eq.globcol}
u_{ij}u_{kl}^*=u_{kl}^*u_{ij}.
\end{equation}
for all possible indices $i,j,k,l$.
\end{defn}

Assuming $G\subset U^+(F)$, this can be equivalently expressed using the entries of the unitary representations $u^{\wcol}=u$ and $u^{\bcol}=F\bar uF^{-1}$ as
\begin{equation}
\label{eq.globcol2}
u_{ij}^{\wcol}u_{kl}^{\bcol}=u_{ij}^{\bcol}u_{kl}^{\wcol}.
\end{equation}

\begin{prop}
\label{P.globcol}
A compact matrix quantum group $G=(C(G),u)$ is globally colourized if and only if for every $w_1, w_2, w_1', w_2'\in\W$ satisfying $|w_1'|=|w_1|$, $|w_2'|=|w_2|$, $c(w_2')-c(w_1')=c(w_2)-c(w_1)$ we have
$$\Mor(u^{\otimes w_1'},u^{\otimes w_2'})=\Mor(u^{\otimes w_1},u^{\otimes w_2}).$$
\end{prop}
\begin{proof}
The equality~\eqref{eq.globcol2} can be also expressed as $u^{\wcol}\otimes u^{\bcol}=u^{\bcol}\otimes u^{\wcol}$, so it is equivalent to saying that the identity is an intertwiner between $u^{\wcol}\otimes u^{\bcol}$ and $u^{\bcol}\otimes u^{\wcol}$. From this, the right-left implication follows directly.

For the left-right implication, from Frobenius reciprocity, it is enough to show the equality for $w_1=w_1'=\emptyset$. It is easy to infer that if the identity is in $\Mor(u^{\wcol}\otimes u^{\bcol},\ppen u^{\bcol}\otimes u^{\wcol})$, we must also have the identity in  $\Mor(u^{\bcol}\otimes u^{\wcol},u^{\wcol}\otimes u^{\bcol})$ and hence also in $\Mor(u^{\otimes w_2},u^{\otimes w_2'})$, and $\Mor(u^{\otimes w_2'},u^{\otimes w_2})$, which implies the desired equality.
\end{proof}

Consider $H\subset O_N^+(F)$. It is easy to check that the tensor complexification $H\tiltimes\hat\Z_k$ is a~globally colourized quantum group with degree of reflection~$k$ for every $k\in\N_0$. In the following theorem, we prove the converse for $k=0$.

\begin{thm}
\label{T.tenscomp}
Consider $G\subset U^+(F)$ with $F\bar F=c\,1_N$, $c\in\R$. Then $G$ is globally colourized with zero degree of reflection if and only if $G=H\tiltimes\hat\Z$, where $H=G\cap O^+(F)$.
\end{thm}
\begin{proof}
We denote by $u$, $v$,~$z$ the fundamental representations of $G$, $H$, and~$\hat\Z_k$, respectively. The quantum group~$H$ is the quantum subgroup of~$G$ defined by the relation $v^{\wcol}=v^{\bcol}$. As mentioned above, the right-left implication is clear since $v_{ij}$ commute with $z$, so
$$u_{ij}^{\wcol}u_{kl}^{\bcol}=v_{ij}zz^*v_{kl}=z^*v_{ij}v_{kl}z=u_{ij}^{\bcol}u_{kl}^{\wcol}.$$

Now, let us prove the left-right implication. First, we show that there is a~surjective $*$\hbox{-}homomorphism
$$\alpha\colon C(G)\to C(H\tiltimes\hat\Z)\subset C(H)\otimes_{\rm max}C^*(\Z)$$
mapping $u_{ij}\mapsto u_{ij}':=v_{ij}z$. To show this, take any element $f\in I_G$. Since $I_G$ is $\Z$-homogeneous, we can assume that $f$ is also $\Z$-homogeneous of some degree~$l$. Then $f(u_{ij}')=f(v_{ij}z)=f(v_{ij})z^l=0$. This proves the existence of such a~homomorphism. Its surjectivity is obvious.

Now it remains to prove that $\alpha$ is injective and hence is a~$*$\hbox{-}isomorphism. Denote by $\xi\in\Mor(1,u^{\wcol}\otimes u^{\bcol})\subset\C^N\otimes\C^N$ the tensor with entries $\xi_{ij}={1\over\sqrt{\Tr(F^*F)}}F_{ji}$, which is normalized so that $\xi^*\xi=1$. We construct a~$*$\hbox{-}homomorphism
$$\beta\colon C(H)\otimes_{\rm max}C^*(\Z)\to M_2(C(G))$$
mapping
$$z\mapsto z':=\begin{pmatrix}0&y\\ 1&0\end{pmatrix},\quad v_{ij}\mapsto v_{ij}':=\begin{pmatrix}0&u_{ij}^{\wcol}\\ u_{ij}^{\bcol}&0\end{pmatrix},\quad y:=\xi^*(u\otimes u)\xi.$$
To prove the existence of such a~homomorphism, we need the following.

Using the fact that $\xi\xi^*\in\Mor(u^{\wcol}\otimes u^{\bcol},u^{\wcol}\otimes u^{\bcol})=\Mor(u\otimes u,u\otimes u)$ (the equality follows from global colourization thanks to Proposition~\ref{P.globcol}), we derive
$$yy^*=\xi^*(u\otimes u)\xi\xi^*(u^*\otimes u^*)\xi=\xi^*\xi\xi^*(uu^*\otimes uu^*)\xi=1$$
and similarly $y^*y=1$. From this, we can also deduce $z'z'^*=z'^*z'=1$.

Using the fact that $1_N\otimes\xi^*\in\Mor(u^{\wcol}\otimes u^{\wcol}\otimes u^{\bcol},u^{\wcol})=\Mor(u^{\bcol}\otimes u^{\wcol}\otimes u^{\wcol},u^{\wcol})$, we derive
$$u^{\bcol}y=(1_N\otimes\xi^*)(u^{\bcol}\otimes u^{\wcol}\otimes u^{\wcol})(1_N\otimes\xi)=u^{\wcol}$$
and similarly $yu^{\bcol}=u^{\wcol}$. This allows us to see that $v_{ij}'z'=z'v_{ij}'=u_{ij}\,1_2$.

Now, it only remains to show that all relations of the generators~$v_{ij}$ are satisfied by~$v_{ij}'$. For this, note that $I_H$ is generated by the relations $v^{\wcol}=v^{\bcol}$ and the ideal~$I_G$. For the first part, we use the assumption $F\bar F=c\,1_N$ to derive
$$v'^{\bcol}:=(1_2\otimes F)\begin{pmatrix}0&\bar u^{\bcol}\\\bar u^{\wcol}&0\end{pmatrix}(1_2\otimes F^{-1})=\begin{pmatrix}0&F\bar F u\bar F^{-1}F^{-1}\\ F\bar u F^{-1}&0\end{pmatrix}=\begin{pmatrix}0&u^{\wcol}\\ u^{\bcol}&0\end{pmatrix}=v'.$$
For the second part, take any $f\in I_G$. Assume it is $\Z$-homogeneous of degree~$i$. Then we have
$$f(v_{ij}')=f(u_{ij}z'^*)=f(u_{ij})z'^{-i}=0.$$

This concludes the proof of existence of~$\beta$. Now, noticing that $\beta\circ\alpha$ is the embedding of~$C(G)$ into diagonal matrices over~$C(G)$, we see that $\alpha$ must be injective.
\end{proof}

\begin{rem}
\label{R.tenscomp}
We leave the situation for general degree of reflection $k\in\N$ open. Modifying the proof, it is actually easy to show that, for any globally colourized~$G$ with degree of reflection~$k$, we have
$$H\tiltimes\hat\Z_k\subset G\subset H\tiltimes\hat\Z.$$
However, we were unable to prove the inclusion $G\subset H\tiltimes\hat\Z_k$. This problem is actually equivalent to proving a stronger version of Proposition~\ref{P.globcol}: Consider $G$ globally colourized with degree of reflection~$k$. Taking $w_1,w_2\in\W$ with $|w_1|=|w_2|$, does it hold that $\Mor(1,u^{\otimes w_1})=\Mor(1,u^{\otimes w_2})$ whenever $c(w_1)\equiv c(w_2)$ modulo~$k$?
\end{rem}

\subsection{Tensor complexification}
\label{secc.tenscomp}

In this section, we study the tensor complexification with respect to representation categories and algebraic relations.

\begin{thm}
\label{T.tenscomp2}
Consider a compact matrix quantum group $G=(C(G),v)$, $k\in\nobreak\N_0$. Denote by $z$ the generator of $C^*(\Z_k)$ and by $u:=vz$ the fundamental representation of $G\tiltimes\hat\Z_k$. We have the following characterizations of $G\tiltimes\hat\Z_k$.
\begin{enumerate}
\item The ideal $I_{G\tiltimes\hat\Z_k}$ is the $\Z_k$-homogeneous part of $I_G$. That is,
$$I_{G\tiltimes\hat\Z_k}=\{f\in I_G\mid f_i\in I_G \text{ for every $i\in\Z_k$}\}.$$
\item The representation category of $G\tiltimes\hat\Z_k$ looks as follows
$$\Mor(u^{\otimes w_1},u^{\otimes w_2})=\begin{cases}\Mor(v^{\otimes w_1},v^{\otimes w_2})&\text{if }c(w_2)-c(w_1)\text{ is a multiple of }k,\\\{0\}&\text{otherwise.}\end{cases}$$
\item $G\tiltimes\hat\Z_k$ is topologically generated by $G$ and $\hat\Z_k$. More precisely, $G\tiltimes\hat\Z_k=\langle G,E\tiltimes\hat\Z_k\rangle$, where $E$ denotes the trivial group of the appropriate size, so $E\tiltimes\hat\Z_k$ is the quantum group $\hat\Z_k$ with the representation $z\oplus\cdots\oplus z=z\, 1_N$.
\end{enumerate}
\end{thm}
\begin{proof}
 For~(1), we can express
$$f(u_{ij},u_{ij}^*)=f(v_{ij}z,z^*v_{ij}^*)=\sum_{l\in\Z_k}f_l(v_{ij}z,z^*v_{ij}^*)=\sum_{l\in\Z_k}f_l(v_{ij},v_{ij}^*)z^l,$$
where $f=\sum_lf_l$ is the decomposition of into the homogeneous components $f_l$ of degree $l$. If all $f_l\in I_G$, so $f_l(v_{ij},v_{ij}^*)=0$, we have $f(v_{ij}z,z^*v_{ij}^*)=0$, so $f\in I_{G\tiltimes\hat\Z_k}$. Conversely, if there is some $l\in\Z_k$ such that $f_l\not\in I_G$, then $f(v_{ij}z,z^*v_{ij}^*)\neq 0$ and hence $f\not\in I_{G\tiltimes\hat\Z_k}$.

For~(2), first we prove that $\Mor(1,u^{\otimes w})=\Mor(1,v^{\otimes w})$ if $c(w)\in k\Z$. Indeed, we have $u^{\otimes w}=(vz)^{\otimes w}=z^{c(w)}v^{\otimes w}=v^{\otimes w}$. Secondly, the fact that $\Mor(1,u^{\otimes w})=\{0\}$ if $c(w)\not\in k\Z$ follows from Lemma~\ref{L.degref}.

To prove (3), note that the category corresponding to $G\tiltimes\hat\Z_k$ (given by~(2)) is indeed the intersection of the category $\Cat_G$ and the category $\Cat_{E\tiltimes\hat\Z_k}$, whose morphism spaces are given by
$$\Mor((z\,1_N)^{\otimes w_1},(z\,1_N)^{\otimes w_2})=\begin{cases}\C^N&\text{if $c(w_2)-c(w_1)$ is a~multiple of $k$},\\\{0\}&\text{otherwise}.\end{cases}$$
\end{proof}

\begin{rem}
An alternative proof of the proposition above could go as follows. One can easily see that the $\Z_k$-extended category associated to $G\times\hat\Z_k$ looks as follows
$$\Cat_{G\times\hat\Z_k}(w_1,w_2)=\begin{cases}\Cat_G(w_1',w_2')&\text{if }t(w_2)-t(w_1)\text{ is a multiple of }k,\\\{0\}&\text{otherwise,}\end{cases}$$
where $w_1',w_2'\in\W$ are created from $w_1,w_2\in\W_k$ mapping $\qcol\mapsto\wcol$, $\Qcol\mapsto\bcol$, $\tcol\mapsto\emptyset$, $\Tcol\mapsto\emptyset$ and by~$t(w)$ we mean the number of white triangles~$\tcol$ minus the number of black triangles~$\Tcol$ in~$w$ (which is a~well-defined element of~$\Z_k$). The item~(2) of the proposition then follows from Proposition~\ref{P.repglue}.
\end{rem}

\begin{rem}
\label{R.itertimes}
As a consequence of Theorem \ref{T.tenscomp2}, we have that
$$(G\tiltimes\hat\Z_k)\tiltimes\hat\Z_l=G\tiltimes\hat\Z_{\lcm(k,l)}$$
A direct proof of this statement was formulated already in \cite[Proposition 8.2]{TW17}.
\end{rem}

\begin{lem}
\label{L.tiltimeszk}
Let $G\subset U^+(F)$ be a~quantum group with degree of reflection~$k$. Denote by $v$ its fundamental representation. Consider $l\in\N_0$ and denote by $z$ the generator of $C^*(\Z_l)$. Then $z^k\in C(G\tiltimes\hat\Z_l)$ for every $l$. Consequently, $z^{nk_0}\in C(G\tiltimes\hat\Z_l)$ for every $n\in\N_0$, where $k_0:=\gcd(k,l)$.
\end{lem}
\begin{proof}
From Proposition~\ref{P.degref}, we can find a~vector $\xi\in\Mor(1,v^{\otimes w})$ with $c(w)=k$ and $\norm{\xi}=1$. Recall that $C(G\tiltimes\hat\Z_l)$ is generated by the elements $v_{ij}z$ and that $v_{ij}$ commute with $z$, so
$$C(G\tiltimes\hat\Z_l)\owns \xi^*(vz)^{\otimes w}\xi=\xi^*v^{\otimes w}\xi\,z^{c(w)}=z^k.$$
Consequently, $z^{nk}\in C(G\tiltimes\hat\Z_k)$ for every $n$ and obviously $\{z^{nk}\}_{n\in\N_0}=\{z^{nk_0}\}_{n\in\N_0}$.
\end{proof}

\begin{prop}
\label{P.tiltimesiso}
Let $G\subset U^+(F)$ be a~quantum group with degree of reflection~$k$. Consider a~number $l\in\N_0$. Then $G\tiltimes\Z_l\simeq G\times\Z_l$ if and only if $k$ is coprime with~$l$.
\end{prop}
\begin{proof}
Assume we have $G\tiltimes\hat\Z_l\simeq G\times\hat\Z_l$. Suppose $d$ is a~divisor of both $k$ and~$l$. Then we must have also $G\tiltimes\hat\Z_d\simeq G\times\hat\Z_d$. But from Lemma~\ref{L.degref}, we have that $G\tiltimes\hat\Z_d=G$, which is a~contradiction unless $d=1$.

For the converse, denote by~$v$ the fundamental representation of~$G$ and by~$z$ the generator of~$C^*(\Z_l)$. It is enough to show that we have $z\in C(G\tiltimes\hat\Z_l)\subset C(G)\otimes_{\rm max}C^*(\Z_l)$ since this already implies the equality of the C*-algebras. This follows directly from Lemma~\ref{L.tiltimeszk}.
\end{proof}

\begin{rem}
If $l$ is not coprime with~$k$, but $l_0:=l/\gcd(k,l)$ is coprime with~$k$, we can use Remark~\ref{R.itertimes}, Lemma~\ref{L.degref}, and Proposition~\ref{P.tiltimesiso} to obtain
$$G\tiltimes\Z_l=(G\tiltimes\Z_{\gcd(k,l)})\tiltimes\Z_{l_0}=G\tiltimes\Z_{l_0}\simeq G\times\Z_{l_0}.$$
\end{rem}

Finally, we are going to characterize irreducible representations of the tensor complexification. Note that the irreducibles of the standard tensor product $G\times\hat\Z_l$ (or $G\times H$ in general) was obtained already by Wang in~\cite{Wan95tensor}.

\begin{prop}
\label{P.tiltimesirrep}
Let $G\subset U^+(F)$ be a quantum group with degree of reflection $k$. Consider arbitrary $l\in\N_0$. Then $G\tiltimes\hat\Z_l$ has the following complete set of mutually inequivalent irreducible representations
\begin{equation}
\label{eq.tiltimesirrep}
\{u^\alpha z^{ki+d_\alpha}\mid \alpha\in\Irr H,\;i=0,\dots,l_0-1\},
\end{equation}
where $l_0=l/\gcd(k,l)$ and $z$ is the generator of $C^*(\Z_k)$.
\end{prop}
\begin{proof}
Since $G$ has degree of reflection~$k$, the ideal $I_G$ is $\Z_k$-homogeneous by Proposition~\ref{P.degref}. This means that the algebra $O(G)$ is $\Z_k$-graded assigning degree one to $v_{ij}$ and degree minus one to $v_{ij}^*$, where $v$ is the fundamental representation of~$G$. Consequently, the entries of any irreducible representation $u^{\alpha}$, $\alpha\in\Irr G$ are $\Z_k$-homogeneous of some degree $d_\alpha$ (recall Sect.~\ref{secc.grading}).

By Theorem~\ref{T.tenscomp2}, $I_{G\tiltimes\hat\Z_l}$ is the $\Z_l$-homogeneous part of $I_G$. Consequently, $I_{G\tiltimes\hat\Z_l}$ is $\Z_{\lcm(k,l)}$-homogeneous and $O(G\tiltimes\hat\Z_l)$ is $\Z_{\lcm(k,l)}$-graded. However, this time the degree is computed with respect to the variables $u_{ij}:=v_{ij}z$.

The irreducible representations of the standard tensor product $G\times\hat\Z_l$ were described by Wang in \cite{Wan95tensor}. Namely, those are exactly all $u^\alpha z^n$ with $\alpha\in\Irr G$, $n=0,\dots,l-1$. We just have to choose those whose matrix entries are elements of $C(G\tiltimes\hat\Z_l)\subset C(G\times\hat\Z_l)$. That is, we need to determine all the pairs $(\alpha,n)$ such that $u^\alpha z^n$ is a~matrix with entries in $C(G\tiltimes\hat\Z_l)\subset C(G\times\hat\Z_l)$.

We first prove that every irreducible of $G\tiltimes\hat\Z_l$ is equivalent to one from Eq.~\eqref{eq.tiltimesirrep}. As we just mentioned, it must be of the form $u^\alpha z^n$ for some $\alpha,n$. Since it is a~representation of $G\tiltimes\hat\Z_l$, it must be a~subrepresentation of $u^{\otimes w}=v^{\otimes w}z^{c(w)}$ for some $w\in\W$. Consequently, $u^\alpha$ is a~subrepresentation of $v^{\otimes w}$, so $d_\alpha\equiv c(w)$ modulo~$k$. In addition, we must also have $n\equiv c(w)$ modulo~$l$. As a~consequence, $n\equiv d_\alpha$ modulo~$k_0:=\gcd(k,l)$. Thus, we must have $n=k_0i+d_\alpha$ for some $i\in\Z$. Obviously, $\{z^{k_0i+d_\alpha}\}_{i\in\Z}=\{z^{ki+d_\alpha}\}_{i=1}^{l_0}$.

For the converse inclusion, we need to show that the entries of $u^\alpha z^{ki+d_\alpha}$ are elements of $C(G\tiltimes\hat\Z_l)$ for every $\alpha,i$. Since $u^\alpha$ is an irreducible representation of~$G$, it must be a~subrepresentation of $v^{\otimes w}$ for some $w\in\W$. Consequently, $u^\alpha z^{c(w)}$ is a~subrepresentation of $u^{\otimes w}=v^{\otimes w}z^{c(w)}$. Hence, it is a~representation of $G\tiltimes\hat\Z_l$. From Lemma~\ref{L.tiltimeszk}, it follows that also $u^\alpha z^{ki+c(w)}$ is a~representation of $G\tiltimes\hat\Z_l$. Since $d_\alpha\equiv c(w)$ modulo~$k$, this is equivalent to considering representations $u^\alpha u^{ki+d_\alpha}$.
\end{proof}

\subsection{Free complexification}
\label{secc.freecomp}

The goal of this section is to characterize the representation categories of the free complexifications, that is, the quantum groups $H\tilstar\hat\Z_l$. For the free complexification, we do not have many results yet even in the easy case. In~\cite{TW17}, the two-coloured categories corresponding to free complexifications of free orthogonal easy quantum groups are provided. For us, the motivating result is \cite[Proposition 4.21]{GW19} linking the free complexification by~$\Z_2$ with the category $\Alt\Cat$ generated by alternating coloured partitions. This proposition was proven with the help of categories of partitions with extra singletons describing the free product with~$\Z_2$ and the functor~$F$ describing the gluing procedure. Also here, we will make use of the $\Z_l$-extended representation categories describing the free product $H*\hat\Z_l$ and then we will glue the factors and apply Proposition~\ref{P.repglue} to find the corresponding representation category. An interesting result is that the free complexification $H\tilstar\hat\Z_l$ actually does not depend on the number $l$ unless the degree of reflection of $H$ equals to one.

\begin{defn}
A monomial of even length of the form $x_{i_1j_1}x^*_{i_2j_2}x_{i_3j_3}x^*_{i_4j_4}\cdots\in\C\langle x_{ij},x^*_{ij}\rangle$, where the variables with and without star alternate, is called \emph{alternating}. A~linear combination of alternating monomials, where either all start with non-star variable or all start with star variable, is called an {\em alternating polynomial}. A~quantum group $G$ is called {\em alternating} if $I_G$ is generated by alternating polynomials.
\end{defn}

Considering a~compact matrix quantum group $G\subset U^+(F)$ with unitary fundamental representation~$u$, recall the notation $u^{\wcol}:=u$, $u^{\bcol}:=F\bar uF^{-1}$. So, the relations of~$G$ can be alternatively expressed by polynomials in variables $u_{ij}^{\wcol}$, $u_{ij}^{\bcol}$ instead of $u_{ij}$ and~$u_{ij}^*$. Since the transformation between those two sets of variables is linear, the definition of an alternating quantum group can be stated in unchanged form also using the alternative ideal.

\begin{lem}
\label{L.Gtilstar0}
Let $H$ be a~compact matrix quantum group, $k,l\in\N_0$ such that $\gcd(k,l)\neq 1$ (using the convention $\gcd(0,k)=k$). Then we have
$$H\tilstar\hat\Z=(H\tiltimes\hat\Z_k)\tilstar\hat\Z_l.$$
\end{lem}
\begin{proof}
We denote by $v$,~$z$, $r$,~$s$ the fundamental representations of $H$,~$\hat\Z$, $\hat\Z_k$, and~$\hat\Z_l$, respectively. We need to find a~$*$\hbox{-}isomorphism $C(H\tilstar\hat\Z)\to C((H\tiltimes\hat\Z_k)\tilstar\hat\Z_l)$ mapping $v_{ij}z\mapsto v_{ij}sr$.

First, we see that there exists a~$*$\hbox{-}homomorphism
$$\alpha\colon C(H)*_\C C^*(\Z)\to (C(H)\otimes_{\rm max}C^*(\Z_k))*_\C C^*(\Z_l)$$
mapping
$$v_{ij}\mapsto v_{ij},\qquad z\mapsto sr.$$
since $sr$ is a~unitary.

This $*$\hbox{-}homomorphism then restricts to a~surjective $*$\hbox{-}homomorphism of the form we are looking for. It remains to prove that it is injective. To prove this, we construct a~$*$\hbox{-}homo\-morphism
$$\beta\colon (C(H)\otimes_{\rm max}C^*(\Z_k))*_\C C^*(\Z_l)\to M_d\bigl(C(H)*_\C C^*(\Z)\bigr)$$
mapping$$
v_{ij}\mapsto\begin{pmatrix}v_{ij}&&0\\&\ddots&\\0&&v_{ij}\end{pmatrix},\quad
s\mapsto\begin{pmatrix}&&&1\\1&&&\\&\ddots&&\\&&1&\end{pmatrix},\quad
r\mapsto\begin{pmatrix}&z^*&&&\\&&1&&\\&&&\ddots&\\&&&&1\\z&&&&\end{pmatrix},
$$
where $d\neq 1$ is some common divisor of $k$ and $l$.

We can check that the images satisfy all the defining relations between the generators, so such a~homomorphism indeed exists. Now, we can see that $\beta\circ\alpha$ is injective, so $\alpha$~must be injective. Thus, the restriction of~$\alpha$ we are interested in is also injective.
\end{proof}

\begin{prop}
\label{P.starind}
Let $H$ be a~compact matrix quantum group and $l\in\N$. Then the $\Z_l$-ex\-tended category $\Cat_{H*\hat\Z_l}$ is generated by the collection $C(\iota(w_1),\iota(w_2)):=\Cat_H(w_1,w_2)$, where $\iota\colon\W\to\W_l$ is the injective homomorphism mapping $\wcol\mapsto\qcol$, $\bcol\mapsto\Qcol$. Moreover, we have the following inductive description. If $w\in\W_k$ contains no triangles, i.e.\ $w=\iota(w')$ for some $w'\in\W$, then
$$\Cat_{H*\hat\Z_l}(\emptyset,w)=\Cat_H(\emptyset,w').$$
Otherwise,
$$\Cat_{H*\hat\Z_l}(\emptyset,w)=\left\{R^{[w_0]}(\xi_1\otimes\cdots\otimes\xi_l)\Biggm|\begin{matrix}w=w_0\tcol w_1\tcol\cdots\tcol w_l\\\xi_i\in\Cat_{H*\hat\Z_l}(\emptyset,w_i),\;i=1,\dots,l-1\\\xi_l\in\Cat_{H*\hat\Z_l}(\emptyset,w_lw_0)\end{matrix}\right\}.$$
\end{prop}
\begin{proof}
Let $\Cat$ be the $\Z_l$-extended category generated by~$C$. Then the associated quantum group~$G=(C(G),v\oplus z)$ is a~quantum subgroup of $U^+(F)*\hat\Z_l$ defined by the relations of~$H$ for~$v$ and no relations for~$z$ (except for $zz^*=z^*z=1=z^l$). But this is exactly the free product $H*\hat\Z_l$. Now, it remains to prove that $\Cat$ is given by the above described recursion.

The inclusion~$\supset$ follows from the fact that $\Cat$ has to be closed under the category operations. To check the inclusion~$\subset$, it is enough to check that the right-hand side defines a~category. That is, we need to check that it is closed under tensor products, contractions, rotations, inverse rotations, and reflections as defined in Sections \ref{secc.Frobenius} and~\ref{secc.prods}. Checking this is straightforward using induction. Nevertheless it may become a~bit lengthy to check all the details. We will do it here for the rotation and tensor product.

So, denote the sets given by the inductive description by~$\tilde\Cat$. If we take words $w$ without triangles, that is, $w=\iota(w')$, then the sets $\tilde\Cat(\emptyset,w)=\Cat_H(\emptyset,w')$ are closed under all the operations since $\Cat_H$ is a~category. To show closedness under rotations in general, we do an induction on the length of the word~$w$. So, consider an element $\xi\in\tilde\Cat(\emptyset,w)$ with $w=w_0\tcol w_1\tcol\cdots\tcol w_l$, so it is of the form $\xi=R^{[w_0]}(\xi_1\otimes\cdots\otimes\xi_l)$. First, suppose that $w_l$ is not empty and denote by~$x$ its last letter. Then we directly have $R\xi=R^{[w_0]+1}(\xi_1\otimes\cdots\otimes\xi_l)\in\tilde\Cat(\emptyset,Rw)$. For the case $w_l=\emptyset$, the last letter of $w$ is a~triangle. So, we need to check that $\tilde\Cat(\emptyset,\tcol w_0\tcol w_1\tcol\cdots\tcol w_{l-1})\owns\xi=(R^{[w_0]}\xi_l)\otimes\xi_1\otimes\cdots\otimes\xi_{l-1}$. This is true thanks to the fact that $R^{[w_0]}\xi_l\in\tilde\Cat(\emptyset,w_0)$ by induction. For the inverse rotations, the proof goes exactly the same way.

Now, we can also prove closedness under the tensor product. Take $\xi\in\tilde\Cat(\emptyset,w)$, $\eta\in\tilde\Cat(\emptyset,w')$. We will do the induction on the length of~$w$. Actually, we can assume that $|w|\ge|w'|$ since we can swap the factors by rotation: $\eta\otimes\xi=R^{[w']}(\xi\otimes R^{-[w']}\eta)$. So, assume $w=w_0\tcol w_1\tcol\cdots\tcol w_l$, so $\xi$ is of the form $\xi=R^{[w_0]}(\xi_1\otimes\cdots\otimes\xi_l)\in\tilde\Cat(\emptyset,w_0\tcol w_1\tcol\cdots\tcol w_l)$.
Then we have
\begin{align*}
                                \xi_l&\in\tilde\Cat(\emptyset,w_lw_0)  &&\text{by assumption,}\\
                         R^{[w_0]}\xi_l&\in\tilde\Cat(\emptyset,w_0w_l)  &&\text{$\tilde\Cat$ closed u.\ rotations,}\\
              R^{[w_0]}\xi_l\otimes\eta&\in\tilde\Cat(\emptyset,w_0w_lw')&&\text{by induction,}\\
\tilde\xi_l=R^{-[w_0]}(R^{[w_0]}\xi_l\otimes\eta)&\in\tilde\Cat(\emptyset,w_lw'w_0)&&\text{$\tilde\Cat$ closed u.\ inv.\ rot.,}\\
\xi\otimes\eta=R^{[w_0]}(\xi_1\otimes\dots\otimes\xi_{l-1}\otimes\tilde\xi_l)&\in\tilde\Cat(\emptyset,w_0\tcol w_1\tcol\cdots\tcol w_lw')&&\text{by definition of $\tilde\Cat$.}
\end{align*}
\end{proof}

\begin{lem}
\label{L.starind}
We can arrange the recursion of the above proposition in such a way that the words $w_1,\dots,w_{l-1}$ contain no triangles, so we have
$$\Cat_{H*\hat\Z_l}(\emptyset,w)=\left\{R^{[w_0]}(\xi_1\otimes\cdots\otimes\xi_l)\Biggm|\begin{matrix}w=w_0\tcol w_1\tcol\cdots\tcol w_l\\\xi_i\in\Cat_H(\emptyset,w_i'),\;i=1,\dots,l-1\\\xi_l\in\Cat_{H*\hat\Z_l}(\emptyset,w_lw_0)\end{matrix}\right\}.$$
\end{lem}
\begin{proof}
We prove this by induction. Take an arbitrary word $w\in\W_l$ and suppose that the above description works for any shorter word. Now consider an element $\xi\in\Cat_{H*\hat\Z_l}(\emptyset,w)$, so it is of the form $\xi=R^{[w_0]}(\xi_1\otimes\cdots\otimes\xi_l)$ corresponding to the decomposition $w=w_0\tcol w_1\tcol\cdots\tcol w_l$. Suppose now that $w_i$ contains some triangles for some $i\in\{1,\dots,{l-1}\}$. By induction hypothesis, we can write $\xi_i=R^{[a_0]}(\eta_1\otimes\cdots\otimes\nobreak\eta_l)$ corresponding to $w_i=a_0\tcol a_1\tcol\cdots\tcol a_l$, where $a_1,\dots,a_{l-1}$ contain no triangles, so $\eta_i\in\Cat_H(\emptyset,a_i')$. But this means that we can write also
$$\xi=R^{[w_0\cdots w_{i-1}a_0]}(\eta_1\otimes\cdots\otimes\eta_{l-1}\otimes\tilde\eta_l),$$
where
\begin{align*}
\tilde\eta_l&=R^{-[a_0]}(R^{[a_0]}\eta_l\otimes\xi_{i+1}\otimes\cdots\otimes\xi_l\otimes R^{-[w_1]}\xi_1\otimes\cdots\otimes R^{-[w_{i-1}]}\xi_{i-1})\\&\in\Cat_{H*\hat\Z_l}(\emptyset,a_l\tcol w_{i+1}\tcol\cdots\tcol w_lw_0\tcol w_1\tcol\cdots\tcol w_{i-1}\tcol a_0).\qedhere
\end{align*}
\end{proof}

In the following theorem, we describe the representation category of the free complexification. In the formulation, we use the following notation. Given an element $w\in\W$ or $w\in\W_l$, we use negative powers to indicate the colour inversion, that is, $w^{-j}=\bar w^j$. For example, $(\wcol\bcol)^{-2}=(\bcol\wcol)^2=\bcol\wcol\bcol\wcol$.

\begin{thm}
\label{T.freecompid}
Let $H$ be a~compact matrix quantum group with degree of reflection $k\neq 1$. Then all $H\tilstar\hat\Z_l$ coincide for all $l\in\N_0\setminus\{1\}$. The ideal $I_{H\tilstar\hat\Z_l}$ is generated by the alternating polynomials in~$I_H$. The representation category $\Cat_{H\tilstar\hat\Z_l}$ is a~(wide) subcategory of the representation category $\Cat_H$ generated by the sets $C(\emptyset,(\wcol\bcol)^j):=\Cat_H(\emptyset,(\wcol\bcol)^j)$, $j\in\Z$. This also holds if $k=1$ and $l=0$.
\end{thm}
\begin{proof}
Let $I\subset\C\langle x_{ij},x_{ij}^*\rangle$ be the ideal generated by the alternating polynomials in~$I_H$. Denote by $u_{ij}=v_{ij}z$ the fundamental representation of $H\tilstar\hat\Z_l$. To prove that $I\subset I_{H\tilstar\hat\Z_l}$, take any alternating polynomial $f\in I_H$. If all monomials in~$f$ start with a~non-star variable, we have $f(v_{ij}z)=f(v_{ij})=0$; if all monomials start with a~star variable, then $f(v_{ij}z)=z^*f(v_{ij})z=0$. In both cases, we have proven that $f\in I_{H\tilstar\hat\Z_l}$. The opposite inclusion $I\supset I_{H\tilstar\hat\Z_l}$ will follow from the statement about representation categories as all the relations corresponding to the elements of~$C$ are alternating.

Note that it is enough to prove the statement for $k\neq 1$ and $l\neq 0$. Indeed, for $k\neq 1$ and $l=0$, we have by Lemma~\ref{L.Gtilstar0} that $H\tilstar\hat\Z=H\tilstar\hat\Z_k$. For $k=1$, $l=0$ we use Lemma~\ref{L.Gtilstar0} to express $H\tilstar\hat\Z=(H\tiltimes\hat\Z_2)\tilstar\Z_l$. Since $c((\wcol\bcol)^j)=0\in 2\Z$ for every~$j$, we have $\Cat_H(\emptyset,(\wcol\bcol)^j)=\Cat_{H\tiltimes\hat\Z_2}(\emptyset,(\wcol\bcol)^j)$.

So, let $\Cat$ be the two-coloured representation category generated by~$C$. We need to prove that $\Cat_{H\tilstar\hat\Z_l}(\emptyset,w)=\Cat(\emptyset,w)$ for every $w\in\W$. In order to do that, we will use Proposition~\ref{P.repglue}, whose statement can be, in this case, formulated as
\begin{equation}
\label{eq.freecompid}
\Cat_{H\tilstar\hat\Z_l}(\emptyset,w)=\Cat_{H*\hat\Z_l}(\emptyset,\tilde w),
\end{equation}
where $\tilde w\in\W_l$ is the glued version of $w\in\W$.

Let us start with the easier inclusion~$\supset$. Since $\Cat_{H\tilstar\hat\Z_l}$ is a~category, it is enough to show that $\Cat_{H\tilstar\hat\Z_l}(\emptyset,w)\supset C(\emptyset,w)$ for every $w=(\wcol\bcol)^j$, $j\in\Z$. Note that the glued version of~$w$ is in this case $\tilde w=(\qcol\tcol\Tcol\Qcol)^j=(\qcol\Qcol)^j$. Combining Proposition~\ref{P.starind} and Equation~\eqref{eq.freecompid}, we have
$$C(\emptyset,(\wcol\bcol)^j)=\Cat_H(\emptyset,(\wcol\bcol)^j)=\Cat_{H*\hat\Z_k}(\emptyset,(\qcol\Qcol)^j)=\Cat_{H\tilstar\hat\Z_l}(\emptyset,(\wcol\bcol)^j).$$

We will prove the opposite inclusion~$\subset$ by induction on the length of~$w$. Take some
$$\xi\in\Cat_{H\tilstar\hat\Z_l}(\emptyset,w)=\Cat_{H*\hat\Z_l}(\emptyset,\tilde w).$$
Suppose $\xi\neq 0$. According to Lemma~\ref{L.starind}, we can assume that $\tilde w=w_0\tcol w_1\tcol\cdots\tcol w_l$, where $w_1,\dots,w_{l-1}$ contain no triangles, and then $\xi=R^{w_0}(\xi_1\otimes\cdots\otimes\xi_l)$ with $\xi_i\in\Cat_H(\emptyset,w_i')$ and $\xi_l\in\Cat_{H*\hat\Z_l}(\emptyset,w_lw_0)$. Since $\tilde w$ is the glued version of~$w$, this means that in all the words $w_1,\dots,w_{l-1}$ the colours alternate (two consecutive white squares would necessarily have a white triangle~$\tcol$ between them, two consecutive black squares would have $\Tcol=\tcol^{l-1}$ between them). Moreover, since we assume $\xi_i\neq 0$, we must have $c(w_i')\in k\Z$ and, since $k\neq 1$, this means that the $w_i$'s are of even length. So, $w_i=(\qcol\Qcol)^{j_i}$, $w_i'=(\wcol\bcol)^{j_i}$.

Finally note that if we delete $\wcol\bcol$ or $\bcol\wcol$ from some word~$w$, its glued version will be given by deleting $\qcol\Qcol$ resp.\ $\Qcol\qcol$. In particular, denote by~$\hat w$ the element~$w$ after deleting all the subwords $w_1',\dots,w_{l-1}'$. Its glued version is then $w_0\tcol^l w_l=w_0w_l$. Using the induction hypothesis, this finishes the proof as we have
$$\xi=R^{w_0}(\xi_1\otimes\cdots\otimes\xi_l)$$
with
\begin{align*}
\xi_i&\in\Cat_H(\emptyset,w_i')=\Cat_H(\emptyset,(\wcol\bcol)^{j_i})=C(\emptyset,(\wcol\bcol)^{j_i})\qquad\text{for $i=1,\dots,l-1$}\\
R^{w_0}\xi_l&\in\Cat_{H*\hat\Z_2}(\emptyset,w_0w_l)=\Cat_{H\tilstar\hat\Z_2}(\emptyset,\hat w)=\Cat(\emptyset,\hat w).\qedhere
\end{align*}
\end{proof}

We may ask what happens if we iterate those free complexifications. The following statement was again already formulated in \cite{TW17}; however, without a~proof. (Note that it generalizes Lemma~\ref{L.Gtilstar0} dropping the assumption $\gcd(k,l)\neq 1$.)

\begin{prop}
Let $H$ be a~compact matrix quantum group, $k,l\in\N_0\setminus\{1\}$. Then
$$(H\tiltimes\hat\Z_k)\tilstar\hat\Z_l=(H\tilstar\hat\Z_k)\tilstar\hat\Z_l=H\tilstar\Z.$$
\end{prop}
\begin{proof}
The second equality follows directly from Theorem~\ref{T.freecompid}~-- we see that iterating the operation on the categories for the second time cannot change it since $\Cat_H(\emptyset,(\wcol\bcol)^j)=\Cat_{H\tilstar\hat\Z_k}(\emptyset,(\wcol\bcol)^j)$. For the first equality, we use, in addition, Theorem~\ref{T.tenscomp2}. Since $c((\wcol\bcol)^j)=\nobreak 0$, we have $\Cat_{H\tiltimes\hat\Z_k}(\emptyset,(\wcol\bcol)^j)=\Cat_H(\emptyset,(\wcol\bcol)^j)$.
\end{proof}

Again, we can ask in what situations does it happen that the glued free product $H\tilstar\hat\Z_l$ is isomorphic to the standard one. Obviously, the necessary condition is that $H$ has degree of reflection one since $H\tilstar\hat\Z_l\simeq H*\hat\Z_l$ implies $H\tiltimes\hat\Z_l\simeq H\times\hat\Z_l$ and here we can use Proposition~\ref{P.tiltimesiso}. We can formulate the converse in the case of globally-colourized quantum groups $H$ (in particular, if $H\subset O^+(F)$).

\begin{prop}
\label{P.Gfreeprodsim}
Let $H$ be a globally colourized compact matrix quantum group with degree of reflection one. Then $H\tilstar\hat\Z_k\simeq H*\hat\Z_k$ for every $k\in\N_0$.
\end{prop}
\begin{proof}
Denote by~$v$ the fundamental representation of~$H$ and by~$z$ the generator of~$C^*(\Z_k)$. Again, it is enough to show that we have $z\in C(H\tiltimes\hat\Z_k)\subset C(H)\otimes_{\rm max}C^*(\Z_k)$ since this already implies the equality of the C*-algebras. From Proposition~\ref{P.degref}, we can find a~vector $\xi\in\Mor(1,v^{\otimes w})$ with $c(w)=1$ and $\norm{\xi}=1$. Since $H$~is globally colourized, we have $\Mor(1,v^{\otimes w})=\Mor(1,v^{\otimes \tilde w})$, where $\tilde w=\wcol\bcol\wcol\bcol\dots\wcol$, $|\tilde w|=|w|$. For such a~word, we have $(vz)^{\otimes\tilde w}=(v^\wcol z)(z^* v^{\bcol})(v^\wcol z)\cdots (v^\wcol z)=v^{\otimes\tilde w}z$, so
\[\xi^*(vz)^{\otimes\tilde w}\xi=\xi^*(v^{\otimes\tilde w}z)\xi=\xi^*\xi\,z=z.\qedhere\]
\end{proof}

\subsection{Free complexification of orthogonal quantum groups}
\label{secc.freecomp2}

In this section, we will study more in detail the free complexification $H*\hat\Z_k$ with $H\subset O^+(F)$. Recall that we define $O^+(F)\subset U^+(F)$ only for $F$ satisfying $F\bar F=c\,1_N$ for some $c\in\R$. We will use this assumption in the whole section.

\begin{defn}
\label{D.colinv}
A~quantum group $G=(C(G),u)\subset U^+(F)$ with $F\bar F=c\,1_N$ is called \emph{invariant with respect to the colour inversion} if the map $u_{ij}\mapsto [F\bar uF^{-1}]_{ij}$ extends to a~$*$-isomorphism.
\end{defn}

Let us explain a~bit this definition. First of all, note that the required $*$-homo\-morphism maps
$$u_{ij}^\wcol\mapsto u_{ij}^{\bcol},\qquad u_{ij}^{\bcol}\mapsto u_{ij}^{\wcol}.$$
Indeed, the first assignment is exactly the definition. For the second one, we have
$$u_{ij}^{\bcol}=[F\bar uF^{-1}]_{ij}\mapsto [F\bar F u\bar F^{-1}F^{-1}]_{ij}=u_{ij}$$
thanks to the assumption $F\bar F=c\,1_N$. In the Kac case $F=1_N$, the homomorphism maps $u_{ij}\mapsto u_{ij}^*$. But let us stress that for general elements of $C(G)$ the homomorphism does not coincide with the $*$-operation (simply because the $*$ is not a~homomorphism).

Secondly, we have the following alternative formulations.

\begin{prop}
\label{P.colinv}
Consider $G=(C(G),u)\subset U^+(F)$ with $F\bar F=c\,1_N$. Then the following are equivalent.
\begin{enumerate}
\item $C(G)$ has an automorphism $u_{ij}^{\wcol}\leftrightarrow u_{ij}^{\bcol}$. That is, $G$ is invariant w.r.t.\ the colour inversion.
\item $I_G$ is invariant w.r.t.\ $x_{ij}^{\wcol}\leftrightarrow x_{ij}^{\bcol}$. More precisely, we mean one of the following equivalent conditions.
\begin{enumerate}
\item $I_G$ is invariant w.r.t.\ the $*$-homomorphism mapping $x_{ij}^{\wcol}\mapsto x_{ij}^{\bcol}$
\item $I_G$ is invariant w.r.t.\ the homomorphism mapping $x_{ij}^{\wcol}\mapsto x_{ij}^{\bcol}$ and $x_{ij}^{\bcol}\mapsto x_{ij}^{\wcol}$.
\end{enumerate}
\item $\Cat_G$ is invariant w.r.t.\ $\wcol\leftrightarrow\bcol$. That is, $\Cat_G(\bar w_1,\bar w_2)=\Cat_G(w_1,w_2)$.
\end{enumerate}
\end{prop}
\begin{proof}
The equivalence $\rm (1)\Leftrightarrow(2a)$ follows from the universal property of $C(G)$.

For $\rm (1)\Rightarrow(3)$, take $T\in\Cat_G(w_1,w_2)$, so $Tu^{\otimes w_1}=u^{\otimes w_2}T$. Applying the automorphism, we get $Tu^{\otimes\bar w_1}=u^{\otimes\bar w_2}T$, so $T\in\Cat_G(\bar w_1,\bar w_2)$.

For $\rm (3)\Rightarrow(2b)$, we use the Tannaka--Krein, namely the fact that $I_G$ is spanned by the relations of the form $Tx^{\otimes w_1}=x^{\otimes w_2}T$. Those relations are invariant with respect to the homomorphism $x^{\wcol}\mapsto x^{\bcol}$, $x^{\bcol}\mapsto x^{\wcol}$ since this homomorphism maps $x^{w}\mapsto x^{\bar w}$. Consequently, the whole ideal $I_G$ must be invariant with respect to this homomorphism.

The implication $\rm (2b)\Rightarrow(1)$ again follows from the universal property of $C(G)$. We get that $u^{\wcol}_{ij}\mapsto u^{\bcol}_{ij}$, $u^{\bcol}_{ij}\mapsto u^{\wcol}_{ij}$ extends to a~homomorphism $C(G)\to C(G)$. Using the assumption $F\bar F=c\,1_N$, we can show this actually must be a~$*$-homomorphism.
\end{proof}

As an example, note that all the universal unitary quantum groups $U^+(F)$ with $F\bar F=c\,1_N$ have this property. In addition, any quantum group $G\subset O^+(F)$ has this property.

\begin{thm}
\label{T.freecompchar}
Consider $G\subset U^+(F)$ with $F\bar F=c\,1_N$. Then $G$ is alternating and invariant with respect to the colour inversion if and only if it is of the form $G=H\tilstar\hat\Z$, where $H=G\cap O^+(F)$.
\end{thm}
\begin{proof}
The right-left implication follows from Theorem~\ref{T.freecompid}: The fact that $H\tilstar\hat\Z$ is alternating is precisely the statement of Theorem~\ref{T.freecompid}. As we mentioned above, $H\subset O^+(F)$ is surely invariant with respect to the colour inversion. According to Proposition~\ref{P.colinv}, this is equivalent to saying that the associated category $\Cat_H$ is invariant with respect to the colour inversion. In particular, we must have $\Cat_H(\emptyset,(\wcol\bcol)^j)=\Cat_H(\emptyset,(\bcol\wcol)^j)$, which are the generators of $\Cat_{H\tilstar\hat\Z}$ according to Theorem~\ref{T.freecompid}. Consequently, also $H\tilstar\hat\Z$ must be invariant with respect to the colour inversion.

In order to prove the left-right implication, we construct a~surjective $*$-homo\-mor\-phism
$$\alpha\colon C(G)\to C(H\tilstar\Z)$$
mapping $u_{ij}\mapsto u_{ij}':=v_{ij}z$. To prove that such a~homomorphism exists, take any alternating element $f\in I_G$. Since $H\subset G$, we have $f(v_{ij})=0$. We need to prove that $f(u_{ij}')=0$. If all terms of $f$ start with a~non-star variable, then $f(u_{ij}')=f(v_{ij}z)=f(v_{ij})=0$; if all terms start with a~star variable, then $f(u_{ij}')=z^*f(v_{ij})z=0$.

It remains to prove that $\alpha$~is injective. To do that, we define a~$*$\hbox{-}homomorphism
$$\beta\colon C(H)*_\C C^*(\Z)\to M_2(C(G))$$
mapping
$$v_{ij}\mapsto v_{ij}':=\begin{pmatrix}0&u_{ij}^{\wcol}\\ u_{ij}^{\bcol}&0\end{pmatrix},\quad z\mapsto z':=\begin{pmatrix}0&1\\ 1&0\end{pmatrix}.$$
We immediately see that indeed $z'z'^*=z'^*z'=1$. In exactly the same way as in the proof of Theorem~\ref{T.tenscomp}, we also prove that $v'^{\bcol}:=(1_2\otimes F)\bar v'(1_2\otimes F^{-1})=v'=:v'^{\wcol}$. Finally, take $f\in I_G$ and, for convenience, use the representation in variables $u_{ij}^{\wcol}$ and~$u_{ij}^{\bcol}$. Suppose $f(x_{ij}^{\wcol},x_{ij}^{\bcol})$ is alternating such that all variables start with~$x_{ij}^{\wcol}$. We have
$$f(v_{ij}',v_{ij}')=\begin{pmatrix}f(u_{ij}^{\wcol},u_{ij}^{\bcol})&0\\ 0&f(u_{ij}^{\bcol},u_{ij}^{\wcol})\end{pmatrix}=0,$$
where $f(u_{ij}^{\wcol},u_{ij}^{\bcol})=0$ directly by $f\in I_G$ and $f(u_{ij}^{\bcol},u_{ij}^{\wcol})$ by invariance under the colour inversion.

Since obviously $\beta\circ\alpha$ is injective, we have proven that $\alpha$ is a~$*$\hbox{-}isomorphism.
\end{proof}

Considering a~quantum group $H=(C(H),v)\subset O^+(F)$, we have $H\tiltimes\hat\Z_k\subset H\tiltimes\hat\Z\subset H\tilstar\hat\Z$. We express those subgroups in terms of relations in the variables $u_{ij}=v_{ij}z$. Of course, those subgroups are given by the relations $v_{ij}z=zv_{ij}$ and $z^k=1$, but those may not be well-defined in $C(H\tilstar\hat\Z)$ as we may not have $z\in C(H\tilstar\hat\Z)$.

\begin{prop}
\label{P.tiltenrel}
Consider $H\subset O^+(F)$. Then $H\tiltimes\hat\Z$ is a~quantum subgroup of $H\tilstar\hat\Z$ given by the relation
\begin{equation}
\label{eq.freetotensor}
u_{ij}^{\wcol}u_{kl}^{\bcol}=u_{ij}^{\bcol}u_{kl}^{\wcol}.
\end{equation}
For $k\in\N$, $H\tiltimes\hat\Z_{2k}$ is a~quantum subgroup of $H\tiltimes\hat\Z$ with respect to the relation
\begin{equation}
\label{eq.0tok}
u_{i_1j_1}^{\wcol}\cdots u_{i_kj_k}^{\wcol}=u_{i_1j_1}^{\bcol}\cdots u_{i_kj_k}^{\bcol}.
\end{equation}
\end{prop}
Before proving the statement, note that Relations \eqref{eq.freetotensor} and~\eqref{eq.0tok} correspond to the two-coloured partitions $\idpart[b/w]\otimes\idpart[w/b]$ and $\idpart[b/w]^{\otimes k}$, respectively. Hence, those are exactly the same relations that were used to construct the quantum groups $H\tiltimes\hat\Z$ and $H\tiltimes\hat\Z_{2k}$ in \cite[Theorem~5.1]{Gro18}.
\begin{proof}
Relation \eqref{eq.freetotensor} is the relation of the global colourization (see Def.~\ref{D.globcol}) and it is obviously satisfied in $H\tiltimes\hat\Z$. We just need to show that imposing this relation is enough. By Corollary~\ref{C.idealgen} and Theorem~\ref{T.tenscomp2}, the ideal $I_{H\tiltimes\hat\Z}$ is generated by relations of the form $u^{\otimes w}\xi=\xi$, where $\xi\in\Cat_{H\tiltimes\hat\Z}(\emptyset,w)=\Cat_H(\emptyset,2l)$, $c(w)=0$, $l:=|w|/2$. In $H\tilstar\hat\Z$, we have a~relation of the form $u^{\otimes (\wcol\bcol)^l}\xi=\xi$. The former relation can surely be derived from the latter one and Relation~\eqref{eq.freetotensor} since it is obtained just by permuting the white and black circles.

The second statement is proven in a~similar way. If we denote $u_{ij}=v_{ij}z$, then Relations~\eqref{eq.0tok} say $v_{i_1j_1}\cdots v_{i_kj_k}z^k=v_{i_1j_1}\cdots v_{i_kj_k}z^{-k}$, which is surely satisfied in $H\tiltimes\hat\Z_{2k}$. For the converse, the ideal $I_{H\tiltimes\hat\Z_k}$ is generated by relations of the form $u^{\otimes w}\xi=\xi$, where $\xi\in\Cat_{H\tiltimes\hat\Z}(\emptyset,w)=\Cat_H(0,2l)$, where $c(w)$ is a~multiple of $2k$ and $l:=|w|/2$. Again, this relation can be derived from $u^{\otimes (\wcol\bcol)^l}\xi=\xi$ using Rel.~\eqref{eq.freetotensor} to permute colours and Rel.~\eqref{eq.0tok} to swap colours of $k$ consecutive white points to black or vice versa.
\end{proof}

\subsection{Free complexification and partition categories}

In sections \ref{secc.globcol}--\ref{secc.tenscomp}, we generalized results that were formulated in the language of categories of partitions in \cite{Gro18}. In contrast, Sections \ref{secc.freecomp}--\ref{secc.freecomp2} were rather new. Hence, it is interesting to look on the special case of easy and non-easy categories of partitions.

We will not recall the theory of partition categories here. Regarding the original definition of categories of partitions as defined in \cite{BS09}, we refer to the survey \cite{Web17}. See also \cite{TW18} for the definition of two-coloured categories of partitions and \cite{GW20} for the definition of linear categories of partitions. Also see \cite{GW19} for the definition of the two-coloured category $\Alt\Cat$ associated to a~given non-coloured category~$\Cat$. Alternatively, everything is summarized in the thesis~\cite{Gro20}.

\begin{prop}
\label{P.AltZ}
Let $\Kat\subset\Part\nlin$ be a~linear category of partitions and denote by $H\subset O_N^+$ the corresponding quantum group. Then $H\tilstar\hat\Z\subset U_N^+$ corresponds to the category $\Alt\Kat$. Moreover, the following holds.
\begin{enumerate}
\item If $0\neq p\in\Kat(0,l)$ for some $l$ odd, then $H\tilstar\hat\Z_k$ corresponds to the category $\langle\Alt\Kat,\tilde p^{\otimes k}\rangle$, where $\tilde p$ is the partition~$p$ with colour pattern $\wcol\bcol\wcol\bcol\cdots\wcol$.
\item If $\Kat(0,l)=\{0\}$ for all $l$ odd, then $H\tilstar\hat\Z_k=H\tilstar\hat\Z$ for all $k\in\N$.
\end{enumerate}
\end{prop}
\begin{proof}
The base statement that $H\tilstar\hat\Z$ corresponds to $\Alt\Kat$ follows directly from Theorem~\ref{T.freecompid}. By Proposition~\ref{P.degrefOG}, the distinction of the cases correspond to the situation that either (1)~$H$~has degree of reflection one or (2)~$H$~has degree of reflection two. So, item~(2) is also contained in Theorem~\ref{T.freecompid}.

For item~(1), denote by $w:=\wcol\bcol\cdots\wcol$ the word of length~$l$ with alternating colours. Since $H$ has degree of reflection one, we have $H\tilstar\hat\Z_m\simeq H*\hat\Z_m$ for every $m\in\N_0$ by Proposition~\ref{P.Gfreeprodsim}. We can actually prove this directly repeating the proof of Prop.~\ref{P.Gfreeprodsim}: Denote by~$v$ the fundamental representation of~$H$ and by~$z$ the fundamental representation of $\hat\Z_m$. Then since we have $v^{\otimes l}\xi_p=\xi_p$, we must have
$$C(H\tiltimes\hat\Z_m)\owns\xi_p^*u^{\otimes w}\xi_p=\xi_p^*(v^{\otimes l}z)\xi_p=z\norm{\xi_p}.$$
Now, $H\tilstar\hat\Z_k$ is just a~quantum subgroup of $H\tilstar\hat\Z$ with respect to the relation $z^k=1$. Note that
$$u^{\otimes w^k}\xi_p^{\otimes k}=(v^{\otimes kl}z^k)\xi_p^{\otimes k}=\xi_p^{\otimes k}z^k.$$
So, the relation $z^k=1$ can be written as $u^{\otimes w^k}\xi_p^{\otimes k}=\xi_p^{\otimes k}$, which is exactly the relation corresponding to~$\tilde p$.
\end{proof}

\begin{prop}
\label{P.AltZeasy}
Let $\Cat\subset\Part$ be a~category of partitions and denote by $G\subset O_N^+$ the corresponding easy quantum group. Then $G\tilstar\hat\Z$ is a~unitary easy quantum group corresponding to the category $\Alt\Cat$.
\begin{enumerate}
\item If $\singleton\not\in\Cat$, then $G\tilstar\hat\Z_k=G\tilstar\hat\Z$ for all $k\in\N$ and it corresponds to the category $\Alt\Cat$.
\item If $\singleton\in\Cat$, then $G\tilstar\hat\Z_k$ corresponds to the category $\langle\Alt\Cat,\singleton[w]^{\otimes k}\rangle$.
\end{enumerate}
\end{prop}
\begin{proof}
This is just a reformulation of Proposition \ref{P.AltZ} to the easy case. Note that $\Cat$ contains some partition of odd length if and only if it contains the singleton~$\singleton$.
\end{proof}

\section{Ungluing}
\label{sec.unglue}

The purpose of this section is to reverse the gluing procedure from Definition~\ref{D.gluedver}. The motivating result is the one-to-one correspondence formulated in terms of partition categories in \cite[Theorem~4.10]{GW19}. According to \cite[Proposition~4.15]{GW19}, the functor~$F$ providing this correspondence translates to the quantum group language exactly in terms of gluing.

Recall from Def.~\ref{D.gluedver} that given a~quantum group $G\subset U^+(F)*\hat\Z_k$ with fundamental representation $u=v\oplus z$, we define its glued version to be the quantum group $\tilde G\subset U^+(F)$ with fundamental representation $\tilde u:=vz$ and underlying C*-algebra $C(\tilde G)\subset C(G)$ generated by the elements $v_{ij}z$.

\begin{defn}
Consider $\tilde G\subset U^+(F)$, $k\in\N_0$. Then any $G\subset U^+(F)* \hat\Z_k$ such that $\tilde G$ is a~glued version of~$G$ is called a~\emph{$\Z_k$-ungluing} of~$\tilde G$.
\end{defn}

In Section~\ref{secc.genunglue}, we are going to study the ungluings in general and show that they always exist. Unsurprisingly, an ungluing of a~quantum group is not defined uniquely. The ungluings introduced in Section~\ref{secc.genunglue} are universal, but not particularly interesting. In Section \ref{secc.OGunglue}, we are going to study more interesting ungluings of the form $G\subset O^+(F)*\hat\Z_2$, which allow us to generalize the above mentioned one-to-one correspondence. We formulate the result as Theorem~\ref{T.gluecorr}, which constitutes the main result of this section.

\subsection{General ungluings}
\label{secc.genunglue}

\begin{prop}
\label{P.unglueex}
There exists a~$\Z_k$-ungluing for every quantum group~$\tilde G$ and for every $k\in\N_0$. Namely, we have the \emph{trivial} $\Z_k$-ungluing $\tilde G\times E$, where $E\subset\hat\Z_k$ is the trivial group. Moreover, $\tilde G\times\hat\Z_k$ is an ungluing of $\tilde G$ whenever $k$ divides the degree of reflection of $G$.
\end{prop}
\begin{proof}
The first statement is obvious as we have $\tilde G\tiltimes E=\tilde G$. The second follows from Lemma~\ref{L.degref} as we have $\tilde G\tiltimes\hat\Z_k=\tilde G$.
\end{proof}

Let us denote by $\iota\colon\C\langle\tilde x_{ij},\tilde x_{ij}^*\rangle\to\C\langle x_{ij},x_{ij}^*,z,z^*\rangle$ the embedding $\tilde x_{ij}\mapsto x_{ij}z$. Consider $\tilde G\subset U^+(F)$ and $G$~its ungluing. The fact that $\tilde G$ is a~glued version of~$G$ is, according to Proposition~\ref{P.repglue}, characterized by the equality $\tilde I_{\tilde G}:=\iota(I_{\tilde G})=I_G\cap \iota(\C\langle\tilde x_{ij},\tilde x_{ij}^*\rangle)$. Consequently, we have $I_G\supset \tilde I_{\tilde G}$ for every ungluing~$G$ of a~quantum group~$\tilde G$.

\begin{defn}
\label{D.maxunglue}
Consider $\tilde G\subset U^+(F)$, $k\in\N_0$. Let $I_G\subset\C\langle x_{ij},x_{ij}^*,z,z^*\rangle$ be the $*$-ideal generated by $\tilde I_{\tilde G}$. Put $C(G):=C^*(\C\langle x_{ij},x_{ij}^*,z,z^*\rangle/I_G)$. Then $G:=(C(G),v\oplus z)$ is called the \emph{maximal $\Z_k$-ungluing} of~$\tilde G$.
\end{defn}

\begin{prop}
The maximal $\Z_k$-ungluing always exists. That is, considering the notation of Definition~\ref{D.maxunglue}, $G$~is indeed a~compact matrix quantum group and $\tilde G$ is indeed its glued version.
\end{prop}
\begin{proof}
First of all, note that $\tilde I_{\tilde G}$ contains the relations $vv^*=v^*v=1_N$ and $v^{\bcol}v^{\bcol *}=v^{\bcol *}v^{\bcol}=1_N$, where $v^{\bcol}=F\bar vF^{-1}$. So, if $G$ is well defined, then we must have $G\subset U^+(F)*\hat\Z_k$.

To prove that $G$ is well defined, we need to check that $I_G/I_{U^+(F)}$ is a Hopf $*$-ideal. Since $I_{\tilde G}/I_{U^+(F)}$ is a Hopf $*$-ideal, we have that $\tilde I_{\tilde G}/I_{U^+(F)}$ is a coideal invariant under the antipode. Consequently, the ideal generated by $\tilde I_{\tilde G}/I_{U^+(F)}$ must be a~Hopf $*$-ideal.

From the construction, it is clear that, if $\tilde G$ has some $\Z_k$-ungluing, then $G$ must be the maximal one (since we take the smallest possible ideal~$I_G$). But every quantum group has the trivial ungluing as mentioned in Prop.~\ref{P.unglueex}.
\end{proof}

\begin{rem}
We do not have to know explicitly the whole ideal $I_{\tilde G}$ to compute the maximal ungluing. Consider $\tilde G\subset U^+(F)$ and suppose that it is determined by a~set of relations $\tilde R$. That is, $I_{\tilde G}$ is generated by the coideal $\tilde R$. Then the maximal $\Z_k$-ungluing $G$ of $\tilde G$ is defined by the relations $R:=\iota(\tilde R)$. That is, taking the generating relations for $\tilde G$ and exchanging $\tilde v_{ij}$ for $v_{ij}z$ and $\tilde v_{ij}^*$ for $z^*v_{ij}^*$.
\end{rem}

Alternatively, we can describe the maximal ungluing by its representation category. Recall the definition of the gluing homomorphism mapping $\wcol\mapsto\qcol\tcol$, $\bcol\mapsto\Tcol\Qcol$. Given a word $w\in\W$, the image $\tilde w$ under this homomorphism is called the {\em glued version} of $w$ by Definition~\ref{D.gluew}. If $G$ is a quantum group and $\tilde G$ its glued version, then $\Cat_{\tilde G}$ is a full subcategory of $\Cat_G$ according to Proposition~\ref{P.repglue}. The full embedding is given exactly by the above mentioned gluing homomorphism. Consequently, the maximal ungluing $G$ of some $\tilde G$ should be a~quantum group with the minimal representation category containing $\Cat_{\tilde G}$ as a~full subcategory.

\begin{prop}
Consider $\tilde G\subset U^+(F)$ and $G\subset U^+(F)*\hat\Z_k$ its maximal $\Z_k$-ungluing. Then the $\Z_k$-extended representation category $\Cat_{G}$ is generated by the sets $C(\tilde w_1,\tilde w_2)=\Cat_{\tilde G}(w_1,w_2)$, where $\tilde w_1,\tilde w_2$ are glued versions of $w_1,w_2\in\W$. In addition, if $\Cat_G$ is generated by some $\tilde C_0$, then $\Cat_{\tilde G}$ is generated by $C_0(\tilde w_1,\tilde w_2)=\tilde C_0(w_1,w_2)$.
\end{prop}
\begin{proof}
By Tannaka--Krein duality, $I_{\tilde G}$ is linearly spanned by relations of the form $[T\tilde v^{\otimes w_1}-v^{\otimes w_2}T]_{\mathbf{ji}}$. By definition of the maximal ungluing, the ideal $I_G$ is generated by elements of $\tilde I_{\tilde G}$, which are exactly the relations $[T u^{\otimes\tilde w_1}-u^{\otimes\tilde w_2}T]_{\mathbf{ji}}$ corresponding to the set $C$.

For the second statement, notice that if $\tilde C_0$ generates $\Cat_{\tilde G}$, then $C_0$ must generate $C$. This follows simply from the fact that gluing of words is a~monoid homomorphism (see also Prop.~\ref{P.repglue}). Consequently, by what was already proven, $C_0$ generates~$\Cat_G$.
\end{proof}

\begin{ex}
As an example, consider the quantum group $\tilde G_k:=O^+(F)\tiltimes\hat\Z_{2k}$, $k\in\N$. By Proposition~\ref{P.tiltenrel}, it is the quantum subgroup of $O^+(F)\tilstar\hat\Z=U^+(F)$ with respect to the relation
$$\tilde v_{i_1j_1}^{\wcol}\tilde v_{i_2j_2}^{\wcol}\cdots\tilde v_{i_kj_k}^{\wcol}=\tilde v_{i_1j_1}^{\bcol}\tilde v_{i_2j_2}^{\bcol}\cdots\tilde v_{i_kj_k}^{\bcol},$$
where $\tilde v$ denotes the fundamental representation of $\tilde G$. We can also take $\tilde G_0:=O^+(F)\tiltimes\hat\Z$, which is a~quantum subgroup of $U^+(F)$ given by
$$\tilde v_{ij}^{\wcol}\tilde v_{kl}^{\bcol}=\tilde v_{ij}^{\bcol}\tilde v_{kl}^{\wcol}.$$

Now, take arbitrary $l\in\N_0$. The maximal $\Z_l$-ungluing of $\tilde G_k$ is a quantum group $G_k\subset U^+(F)*\hat\Z_l$ with fundamental representation of the form $v\oplus z$ given by the same relations if we substitute $\tilde v_{ij}$ by $v_{ij}z$.
$$v_{i_1j_1}^{\wcol}zv_{i_2j_2}^{\wcol}z\cdots v_{i_kj_k}^{\wcol}z=z^*v_{i_1j_1}^{\bcol}z^*v_{i_2j_2}^{\bcol}\cdots z^*v_{i_kj_k}^{\bcol}.$$
The ungluing $G_0$ is defined by the relation
$$v_{ij}^{\wcol}v_{kl}^{\bcol}=z^*v_{ij}^{\bcol}v_{kl}^{\wcol}z.$$
We can also represent the relations diagrammatically. The defining relations of $O^+(F)\tiltimes\hat\Z_k$ and $O^+(F)\tiltimes\hat\Z$ are $(\idpart[b/w])^{\otimes k}$ and $\globcol$ respectively (see also \cite{TW17,Gro18}). To obtain the maximal ungluing, we can have to put a white triangle $\tcol$ after every white circle $\wcol$ and a black triangle $\Tcol$ in front of every black circle $\bcol$. So, the defining relation for $G_k$ corresponds to $(\Partition{\Pline (1,1)(2,0) \LT(1,0) \Ut(2,1)}[0b/w])^{\otimes k}$, for $G_0$, we have $\Partition{\Pline(2,1)(2,0) \Pline(3,1)(3,0) \LT(1,0) \Lt(4,0)}[0bw/0wb]$.

We can see that the result is quite something different than simply $O^+(F)\times\hat\Z_{2k}$ as we may have hoped.
\end{ex}

In general, the maximal ungluing never provides any useful decomposition into smaller pieces since we always have the trivial decomposition inside the maximal one $\tilde G\times E\subset G$.

\subsection{Orthogonal ungluings}
\label{secc.OGunglue}

As just mentioned, constructing large unitary ungluings $G\subset U^+(F)*\hat\Z_k$ may not be very useful. In this subsection, we study ungluings that are orthogonal, that is, of the form $G\subset O^+(F)*\hat\Z_2$. For the rest of this subsection, we assume $F\bar F=c\,1_N$.

For a~given quantum group $G\subset O^+(F)* \hat\Z_2$, we will denote by~$I_G$ the corresponding ideal inside $A:=\C\langle x_{ij}\rangle*\C\Z_2$ (instead of taking $\C\langle x_{ij},r\rangle$ or $\C\langle x_{ij},x_{ij}^*,r,r^*\rangle$). The generator of~$\C\Z_2$ will be denoted by~$r$. Note that we have to consider the non-standard involution $x_{ij}^*=[F^{-1}xF]_{ij}$ on~$A$. The algebra $A$ is $\Z_2$-graded (assigning all variables $x_{ij}$ and~$r$ degree one). We will denote by $\tilde A$ the even part of~$A$. Then $\tilde Ar$ is the odd part of~$A$.

\begin{lem}
The mapping $x_{ij}\mapsto x_{ij}r$ extends to an injective $*$\hbox{-}homomorphism $\iota_A\colon\C\langle x_{ij},x_{ij}^*\rangle\to A$. Its image $\iota_A(A)$ equals to $\tilde A$~-- the even part of $A$.
\end{lem}
\begin{proof}
The existence of the $*$\hbox{-}homomorphism $\iota_A$ follows from the fact that its domain is a~free algebra. Obviously, for any monomial $f\in\C\langle x_{ij},x_{ij}^*\rangle$, the image $\iota_A(f)$ has even degree. It remains to show that, for any monomial of even degree $\tilde f\in\tilde A$, there exists a~unique monomial $f\in\C\langle x_{ij},x_{ij}^*\rangle$ such that $\tilde f=\iota_A(f)$.

This is done easily by induction on the ``length'' of $\tilde f$ measured by the number of variables $x_{ij}$ or~$x_{ij}^*$ (ignoring the~$r$'s). Suppose $\tilde f$ is in the reduced form, that is, the variable~$r$ does not appear twice consecutively. If $\tilde f$ starts with the variable~$x_{ij}$, we can write $\tilde f(x_{ij},r)=x_{ij}r\tilde f_0(x_{ij},r)$ for some $f_0\in\tilde A$, so $\iota_A^{-1}(\tilde f)(x_{ij},x_{ij}^*)=x_{ij}\iota_A^{-1}(\tilde f_0)(x_{ij},x_{ij}^*)$. If $\tilde f$ starts with~$r$ followed by~$x_{ij}$, so $\tilde f(x_{ij},r)=rx_{ij}\tilde f_0(x_{ij},r)$ for some $f_0\in\tilde A$, then $\iota_A^{-1}(\tilde f)(x_{ij},x_{ij}^*)=x_{ij}^*\iota_A^{-1}(\tilde f_0)(x_{ij},x_{ij}^*)$.
\end{proof}

\begin{rem}
\label{R.glue}
$\tilde A$~is the $*$-subalgebra of~$A$ generated by the elements~$x_{ij}r$. Consequently, for any $G\subset O^+(F)*\hat\Z_2$, we can express the coordinate algebra associated to its glued version $\tilde G$ as
$$O(\tilde G)=\{f(v_{ij},r)\mid f\in\tilde A\}\subset O(G).$$
In addition, we can rephrase Proposition~\ref{P.repglue} by saying that a~quantum group $\tilde G\subset U^+(F)$ is a~glued version of $G\subset O^+(F)* \hat\Z_2$ if and only if we have $\tilde I_{\tilde G}:=\iota_A(I_{\tilde G})=I_G\cap\tilde A$.
\end{rem}

Recall the definition of quantum groups $\tilde G\subset U^+(F)$ invariant with respect to the colour inversion from Def.~\ref{D.colinv}.

\begin{prop}
\label{P.canon}
Consider a~compact matrix quantum group $\tilde G\subset U^+(F)$ with $F\bar F=c\,1_N$ invariant with respect to the colour inversion. Let $G'$ be its maximal $\Z_2$-ungluing. Then $G:=G'\cap(O^+(F)* \hat\Z_2)$ is also a~$\Z_2$-ungluing.
\end{prop}

\begin{defn}
\label{D.canon}
The quantum group $G$ from Proposition~\ref{P.canon} will be called the \emph{canonical $\Z_2$-ungluing} of~$\tilde G$.
\end{defn}

Before proving the proposition, recall that $I_{G'}\subset\C\langle x_{ij},x_{ij}^*,\ppen r,r^*\rangle$ is defined as the smallest ideal containing $\iota(I_{\tilde G})$. Consequently, $I_G=I_{G'}/{(x^{\wcol}=x^{\bcol}, r^2=1)}$ is the smallest ideal of the algebra $A=\C\langle x_{ij},x_{ij}^*,r,r^*\rangle/{(x^{\wcol}=x^{\bcol}, r^2=1)}$ containing $\tilde I_{\tilde G}=\iota(I_{\tilde G})/{(x^{\wcol}=x^{\bcol}, r^2=1)}$. In other words, the canonical $\Z_2$-ungluing is determined by relations of the form $f(x_{ij}r,rx_{ij})$ with $f\in I_{\tilde G}\subset\C\langle x_{ij},x_{ij}^*\rangle$.

\begin{proof}
Adopting the notation introduced above, we need to prove that $I_G\cap\tilde A=\tilde I_{\tilde G}$. We prove that
$$I_G=\tilde I_{\tilde G}+\tilde I_{\tilde G}r=\spanlin\{f,fr\mid f\in\tilde I_{\tilde G}\}.$$
Then it will be clear that $I_G\cap\tilde A=(\tilde I_{\tilde G}+\tilde I_{\tilde G}r)\cap\tilde A=\tilde I_{\tilde G}$ since $\tilde I_{\tilde G}\subset\tilde A$, so $\tilde I_{\tilde G}r\cap\tilde A=\emptyset$.

To prove the equality, it is enough to show that the right-hand side is an ideal since then it must be the smallest one containing $\tilde I_{\tilde G}$, which is exactly~$I_G$. So, denote the right-hand side by~$I$. By Proposition~\ref{P.colinv}, $G$ being invariant with respect to the colour inversion means that $I_G$ is invariant with respect to interchanging $x_{ij}^{\wcol}\leftrightarrow x_{ij}^{\bcol}$. Applying $\iota_A$, we get that $\tilde I_{\tilde G}$ is invariant with respect to $x_{ij}r\leftrightarrow rx_{ij}$, so $\tilde I_{\tilde G}$ is invariant with respect to conjugation by~$r$, that is, $x\mapsto rxr$. We use that to prove that $I$ is an ideal. The subspace~$I$ is obviously invariant under right multiplication by~$r$. For the left multiplication, we can write $rx=(rxr)r$. For multiplication by~$x_{ij}$, we can write $xx_{ij}=(xx_{ij}r)r$ and $x_{ij}x=r((x_{ij}r)^{\bcol}x)$.
\end{proof}

\begin{prop}
\label{P.orthglue}
Consider $G\subset O^+(F)*\hat\Z_2$ and denote by~$k$ its degree of reflection. Then exactly one of the following situations occurs.
\begin{enumerate}
\item If $k=1$, then $C(\tilde G)=C(G)$ and hence $\tilde G\simeq G$.
\item If $k=2$, then $C(G)$ is $\Z_2$-graded and $C(\tilde G)$ is its even part.
\end{enumerate}
\end{prop}
\begin{proof}
Since $G$ is orthogonal, its degree of reflection must be either one or two by Proposition~\ref{P.degrefOG}. First, let us assume that $k=1$. To show that $C(\tilde G)=C(G)$, it is enough to prove that $r\in C(\tilde G)$. The assumption $k=1$ means that there is a~vector $\xi\in\Mor(1,u^{\otimes k})$, $\norm{\xi}=1$ for some $k$ odd, so we have $\xi^*u^{\otimes k}\xi=1$ in~$C(G)$. Consequently, $r=\xi^*u^{\otimes k}\xi r$ holds in~$C(G)$, and, since it is an even polynomial, it must be an element of~$C(\tilde G)$ (see Remark~\ref{R.glue}).

If the degree of reflection equals two, then by Proposition~\ref{P.degref} the $\Z_2$-grading of~$A$ passes to~$O(G)$ and hence also~$C(G)$. As mentioned in Remark~\ref{R.glue}, $O(\tilde G)$~consists of even polynomials in the generators $v_{ij}$ and~$r$ and hence is the even part of~$O(G)$.
\end{proof}

The following theorem provides a~non-easy counterpart of Theorem \cite[Theorem~4.10]{GW19}.

\begin{thm}
\label{T.gluecorr}
There is a~one-to-one correspondence between
\begin{enumerate}
\item quantum groups $G\subset O^+(F)*\hat\Z_2$ with degree of reflection two and
\item quantum groups $\tilde G\subset U^+(F)$ that are invariant with respect to the colour inversion.
\end{enumerate}
This correspondence is provided by gluing and canonical $\Z_2$-ungluing.
\end{thm}
\begin{proof}
Almost everything follows from Proposition~\ref{P.canon}. The only remaining thing to prove is that, given $G\subset O^+(F)*\hat\Z_2$ and $\tilde G$~its glued version, then $\tilde G$ is invariant with respect to the colour inversion and $G$~is its canonical $\Z_2$-ungluing. The first assertion follows from the fact that $I_G$ and hence also $\tilde I:=I_G\cap\tilde A=\tilde I_{\tilde G}$ is invariant with respect to conjugation by~$r$. For the second assertion, we need to prove that $I_G=\tilde I+\tilde Ir$ (see the proof of Proposition~\ref{P.canon}). Since $G$ has degree of reflection two, we have that $I_G$ is $\Z_2$-graded and hence it decomposes into an even and odd part, which is precisely $\tilde I$ and~$\tilde Ir$.
\end{proof}

In~\cite{GW19}, we introduced the $\Z_2$-extensions $H\times_{2k}\hat\Z_2$ exactly to be the canonical $\Z_2$-ungluings of $H\tiltimes\hat\Z_{2k}$. In the following, we recall the definition of those products and provide a~direct proof for the fact that $H\times_{2k}\hat\Z_2$ are canonical $\Z_2$-ungluings of $H\tiltimes\hat\Z_{2k}$ for arbitrary $H\subset O^+(F)$.

\begin{defn}[\cite{GW19}]
\label{D.prods}
Let $G$ and $H$ be compact matrix quantum groups and denote by $u$ and $v$ their respective fundamental representations. We define the following quantum subgroups of $G*H$. The product $G\ttimes H$ is defined by taking the quotient of $C(G*H)$ by the relations
\begin{equation}
\label{eq.ttimes}
ab^*x=xab^*,\qquad a^*bx=xa^*b
\end{equation}
the product $G\timess H$ is defined by the relations
\begin{equation}
\label{eq.timess}
ax^*y=x^*ya,\qquad axy^*=xy^*a
\end{equation}
the product $G\times_0 H$ by the combination of the both pairs of relations and, finally, given $k\in\N$, the product $G\times_{2k}H$ is defined by the relations
\begin{equation}
\label{eq.times}
a_1x_1\cdots a_kx_k=x_1a_1\cdots x_ka_k,
\end{equation}
where $a,b,a_1,\dots,a_k\in\{u_{ij}\}$ and $x,y,x_1,\dots,x_k\in\{v_{ij}\}$. (Equivalently, we can assume $a,b,a_1,\dots,a_k\in\spanlin\{u_{ij}\}$ and $x,y,x_1,\dots,x_k\in\spanlin\{v_{ij}\}$.)
\end{defn}

\begin{prop}
\label{P.timeskcanon}
Consider $H\subset O^+(F)$ with degree of reflection two, $k\in\N_0$. Then the canonical $\Z_2$-ungluing of $H\tiltimes\hat\Z_{2k}$ is $H\times_{2k}\hat\Z_2$.
\end{prop}
\begin{proof}
From Proposition \ref{P.tiltenrel}, we can express $H\tiltimes\hat\Z_{2k}$ as a quantum subgroup of $H\tilstar\hat\Z=H\tilstar\hat\Z_2$ given by certain relations. Denoting by $\tilde v$ the fundamental representation of $H\tiltimes\hat\Z_{2k}$, we just have to ``unglue'' the relations substituting $\tilde v_{ij}$ by $v_{ij}r$. For $H\tiltimes\hat\Z$, we get
$$v_{ij}v_{kl}=rv_{ij}v_{ij}r,$$
which is obviously equivalent to $v_{ij}v_{kl}r=rv_{ij}v_{kl}$ -- the defining relation for $H\ttimes\hat\Z_2=H\times_0\hat\Z_2$. For $H\tiltimes\hat\Z_{2k}$, $k>0$, we get
$$v_{i_1j_1}zv_{i_2j_2}z\cdots v_{i_kj_k}z=zv_{i_1j_1}zv_{i_2j_2}\cdots zv_{i_kj_k},$$
which is exactly the defining relation for $H\times_{2k}\hat\Z_2$.
\end{proof}

\subsection{Irreducibles and coamenability for ungluings and $\Z_2$-extensions}
In this section, we will study properties of the gluing procedure and canonical $\Z_2$-ungluing. In many cases, we can view the $\Z_2$-extensions $H\times_{2k}\hat\Z_2$ as a motivating example. It remains an open question, whether one can generalize the statements for arbitrary products $H_1\times_{2k}H_2$.

For this section, we will assume that $G\subset O_N^+*\hat\Z_2$ has degree of reflection two.

For quantum groups $G$ and $H$, the property $O(H)\subset O(G)$ can be understood either as $H$ being a~quotient of $G$ or the discrete dual $\hat H$ being a~quantum subgroup of~$\hat G$. In that case, we can study the homogeneous space $\hat G/\hat H$ by defining
$$l^{\infty}(\hat G/\hat H):=\{x\in l^{\infty}(\hat G)\mid x(ab)=x(b)\hbox{ for all $a\in O(H)$ and $b\in O(G)$}\},$$
where $l^{\infty}(\hat G)$ is the space of all bounded functionals on~$O(G)$.

\begin{prop}
Consider $G\subset O_N^+*\Z_2$ with degree of reflection two. Then $\hat G/\hat{\tilde G}$ consists of two points. More precisely,
$$l^\infty(\hat G/\hat{\tilde G})=\{x\in l^{\infty}(\hat G)\text{ constant on the $\Z_2$-homogeneous parts of $O(G)$}\}\simeq\C^2.$$
\end{prop}
\begin{proof}
Consider $x\in l^\infty(\hat G)$. By Proposition~\ref{P.orthglue}, $O(G)$ is $\Z_2$-graded. Putting $b:=1$ in the equality $x(ab)=x(b)$, we get that $x$ is constant on the even part. Putting $b:=r$, we get that $x$ is constant on the odd part.
\end{proof}

Recall that a compact quantum group $G$ is called \emph{coamenable} if the so-called \emph{Haar state} is faithful on the universal C*-algebra $C(G)$.

\begin{prop}
\label{P.unglueamen}
Consider $G\subset O_N^+*\hat\Z_2$ and $\tilde G\subset U_N^+$ its glued version. Then $G$ is coamenable if and only if $\tilde G$ is coamenable.
\end{prop}
\begin{proof}
It is easy to see that the Haar state on $\tilde G$ is given just by restriction of the Haar state $h$ on $G$. We also have that all positive elements of $C(G)$ are contained in $C(\tilde G)$. Hence, the faithfulness of $h$ on $C(G)$ is equivalent to the faithfulness on $C(\tilde G)$.
\end{proof}

Note that it is well known that the coamenability is preserved under tensor product of quantum groups, but it is not preserved under the (dual) free product. Hence, it is interesting to ask, whether it is preserved under the new interpolating products from Def.~\ref{D.prods}.

Let us also recall another well known fact that can be seen directly from the definition of coamenability: Let $G$ be a~coamenable compact quantum group. Then every its quotient, that is, every quantum group $H$ with $C(H)\subset C(G)$ is coamenable.

\begin{prop}
Consider $H\subset O_N^+$, $k\in\N_0$. Then $H\times_{2k}\hat\Z_2$ is co-amenable if and only if $H$ is co-amenable.
\end{prop}
\begin{proof}
The left-right implication follows from the fact that $H$ is a~quotient of $H\times_{2k}\hat\Z_2$. Now, let us prove the right-left implication. If the degree of reflection of~$H$ is one, then by \cite[Theorem~5.5]{GW19}, the new product actually coincides with the tensor product $H\times_{2k}\hat\Z_2=H\times\hat\Z_2$, so it indeed preserves the coamenability.

Now suppose $H$ has degree of reflection two. If $H$ is coamenable, then $H\times\hat\Z_{2k}$ must be coamenable. Consequently, its glued version $H\tiltimes\hat\Z_{2k}$ is coamenable since it is a~quotient quantum group. Finally, $H\times_{2k}\hat\Z_2$ is coamenable by Proposition~\ref{P.unglueamen}.
\end{proof}

Now, we are going to look on the irreducible representations of the ungluings.

\begin{prop}
\label{P.unglueirrep}
Consider $G\subset O^+(F)*\hat\Z_2$ with degree of reflection two and fundamental representation $v\oplus r$. Let $\tilde G\subset U_N^+$ be its glued version. Then the irreducibles of~$G$ are given by
$$\{u^\alpha,u^\alpha r\mid\alpha\in\Irr\tilde G\}.$$
\end{prop}
\begin{proof}
First, we prove that all the matrices are indeed representations of~$G$. Surely all $u^\alpha$ are representations. The $r$ is also a~representation. Hence $u^\alpha r=u^\alpha\otimes r$ must also be representations.

Secondly, we prove that the representations are mutually inequivalent. The representations~$u^\alpha$ are mutually inequivalent by definition. From this, it follows that the representations~$u^\alpha r$ are mutually inequivalent. Since $G$ has degree of reflection two, we have that $O(G)$ is graded with $O(\tilde G)$ being its even part. So, the entries of $u^\alpha$ are even, whereas the entries of $u^\alpha r$ are odd, so they cannot be equivalent.

Finally, we need to prove that those are all the representations. This can be proven using the fact that entries of irreducible representations form a~basis of the polynomial algebra. If we prove that the entries of the representations span the whole $O(G)$, we are sure that we have all the irreducibles. This is indeed true:
\[\spanlin\{u^\alpha_{ij},u^\alpha_{ij}r\mid\alpha\in\Irr\tilde G\}=O(\tilde G)+O(\tilde G)\,r=O(G).\qedhere\]
\end{proof}

\begin{prop}
Consider $H\subset O^+(F)$ with degree of reflection two, $k\in\N_0$. Then the complete set of mutually inequivalent irreducible representations of $H\times_{2k}\hat\Z_2$ is given by
\begin{equation}
u^{\alpha,i,\eta}=u^\alpha s^ir^\eta,\qquad \alpha\in\Irr H,\;i\in\{0,\dots,k-1\},\;\eta\in\{0,1\},
\end{equation}
where 
\begin{equation}
\label{eq.srep}
s=\sum_l v_{il}rv_{il}^*r=\sum_k v_{kj}^*rv_{kj}r\quad\hbox{for any $i,j=1,\dots,N$.}
\end{equation}
Here $v\oplus r$ denotes the fundamental representation of $H\times_{2k}\hat\Z_2$.
\end{prop}
\begin{proof}
Denote $\tilde v:=vr$ the fundamental representation of the glued version of $H\times_{2k}\hat\Z_2$, which can be identified with $H\tiltimes\hat\Z_{2k}$ by Proposition~\ref{P.timeskcanon}. Thus, we can also write $\tilde v_{ij}=v_{ij}z\in C(H\tiltimes\hat\Z_{2k})\subset C(H\times\hat\Z_{2k})$. We can also express
\begin{equation}
\begin{split}
s&=\sum_l v_{il}rv_{il}^*r=[vrvr^*]_{ii}=[vr(F^{-1}vrF)^t]_{ii}\\&=[\tilde v(F^{-1}\tilde vF)^t]_{ii}=[v(F^{-1}vF)^tz^2]_{ii}=[vv^*]_{ii}z^2=z^2
\end{split}
\end{equation}
and similarly for the second expression in Eq.~\eqref{eq.srep}. In particular, we have that $s=z^2$ is a~representation of $H\tiltimes\hat\Z_{2k}$ and hence also of $H\times_{2k}\Z_2$.

According to Proposition~\ref{P.tiltimesirrep}, we know that irreducibles of $H\tiltimes\hat\Z_{2k}$ are of the form $u^{\alpha,i}:=u^\alpha z^{2i+d_\alpha}$, $\alpha\in\Irr H$, $i=0,\dots,k-1$ and $d_\alpha\in\{0,1\}$ is the degree of~$\alpha$. According to Proposition~\ref{P.unglueirrep}, we have that the set of irreducibles of $H\times_{2k}\hat\Z_{2k}$ is $u^{\alpha,i}r^\eta$. Now the only point is to express these in terms of $u^\alpha$, $s$, and~$r$.

Suppose first that $d_\alpha=0$. In this case, the situation is simple since we can use the above mentioned fact that $z^2=s$ to derive
$$u^{\alpha,i}r^\eta=u^\alpha z^{2i}r^\eta=u^\alpha s^ir^\eta.$$

In the situation $d_\alpha=1$, we need to prove that $u^\alpha z=u^\alpha r$. The left hand side is a~representation of $H\tiltimes\hat\Z_{2k}$~-- the glued version of $H\times\hat\Z_{2k}$~-- and the right-hand side is a~representation of the glued version of $H\times_{2k}\hat\Z_2$. As we already mentioned, these quantum groups coincide, so the equality makes sense. (Note that it is not possible to show that $z=r$. Not only that this is not true, the equality does not even make sense since $s$ and~$r$ are not elements of a common algebra.) Since $u^\alpha z$ is a~representation of $H\tiltimes\hat\Z_{2k}$, we have, in particular, that the entries of the representation $u^\alpha_{ab}z$ are elements of $O(H\tiltimes\hat\Z_{2k})$. That is, there are polynomials $f^\alpha_{ab}\in\C\langle x_{ij}^{\wcol},x_{ij}^\bcol\rangle$ of degree one such that $f^\alpha_{ab}(\tilde v_{ij}^\wcol,\tilde v^\bcol_{ij})=f^\alpha_{ab}(v_{ij}z,z^*v_{ij})=u^\alpha_{ab}z$. Since $v_{ij}$ commute with $z$, we can arrange the ``$\wcol\bcol$-pattern'' of the monomials in $f^\alpha_{ab}$ in an arbitrary way if we keep the property that they have degree one. In particular, we can say that every monomials of $f^\alpha_{ab}$ have an alternating colour pattern, that is, they are of the form $x_{i_1j_1}^\wcol x_{i_2j_2}^\bcol\cdots x_{i_nj_n}^\wcol$. Then we can express
$$u^\alpha_{ab}z=f^\alpha_{ab}(v_{ij}z,z^*v_{ij})=f^\alpha_{ab}(\tilde v_{ij}^\wcol,\tilde v^\bcol_{ij})=f^\alpha_{ab}(v_{ij}r,rv_{ij})=u^\alpha_{ab}r.$$

This is exactly what we wanted to prove. Now we have
$$u^{\alpha,i}r^\eta=u^\alpha zz^ir^\eta=u^\alpha rs^ir^\eta$$
To obtain the form in the statement, note just that $rsr=s^*=s^{k-1}$.
\end{proof}

\bibliographystyle{halpha}
\bibliography{mybase}

\end{document}